\documentclass[a4]{amsart}

\input xypic 
\input xy 
\xyoption{all} 
\usepackage{amssymb} 
\usepackage{bbm}
\usepackage{tikz-cd}
\newtheorem{theorem}{Theorem}[section]
\newtheorem{proposition}[theorem]{Proposition}
\newtheorem{lemma}[theorem]{Lemma}
\newtheorem{corollary}[theorem]{Corollary}

\theoremstyle{definition}

\newtheorem{remark}[theorem]{Remark}

\newcommand{\hlgy}[1]{\ensuremath{H_{*}(#1)}}

\newcommand{\cohlgy}[1]{\ensuremath{H^{*}(#1)}}

\newcounter{bean}
\newenvironment{letterlist}{\begin{list}{\rm ({\alph{bean}})}
      {\usecounter{bean}\setlength{\rightmargin}{\leftmargin}}}
      {\end{list}}

\newcommand{\namedright}[3]{\ensuremath{#1\stackrel{#2}
 {\longrightarrow}#3}}
\newcommand{\nameddright}[5]{\ensuremath{#1\stackrel{#2}
 {\longrightarrow}#3\stackrel{#4}{\longrightarrow}#5}}
\newcommand{\namedddright}[7]{\ensuremath{#1\stackrel{#2}
 {\longrightarrow}#3\stackrel{#4}{\longrightarrow}#5
  \stackrel{#6}{\longrightarrow}#7}}

\newcommand{\larrow}{\relbar\!\!\relbar\!\!\rightarrow}
\newcommand{\llarrow}{\relbar\!\!\relbar\!\!\larrow}
\newcommand{\lllarrow}{\relbar\!\!\relbar\!\!\llarrow}

\newcommand{\lnameddright}[5]{\ensuremath{#1\stackrel{#2}
 {\larrow}#3\stackrel{#4}{\larrow}#5}}
\newcommand{\lnamedddright}[7]{\ensuremath{#1\stackrel{#2}
 {\larrow}#3\stackrel{#4}{\larrow}#5
  \stackrel{#6}{\larrow}#7}}
\newcommand{\llnamedright}[3]{\ensuremath{#1\stackrel{#2}
 {\llarrow}#3}}
\newcommand{\llnameddright}[5]{\ensuremath{#1\stackrel{#2}
 {\llarrow}#3\stackrel{#4}{\llarrow}#5}}

\newcommand{\lllnamedright}[3]{\ensuremath{#1\stackrel{#2}
 {\lllarrow}#3}}
\newcommand{\lllnameddright}[5]{\ensuremath{#1\stackrel{#2}
 {\lllarrow}#3\stackrel{#4}{\lllarrow}#5}}

\newcommand{\qqed}{\hfill\Box}

\begin{document}


\title[Loop spaces of Poincar\'{e} Duality complexes]
   {Loop space decompositions of $(2n-2)$-connected $(4n-1)$-dimensional 
   Poincar\'{e} Duality complexes} 

\author{Ruizhi Huang} 
\address{Institute of Mathematics, Academy of Mathematics and Systems Science, 
   Chinese Academy of Sciences, Beijing 100190, China} 
\email{haungrz@amss.ac.cn} 
   \urladdr{https://sites.google.com/site/hrzsea/}
\author{Stephen Theriault}
\address{School of Mathematics, University of Southampton, Southampton 
   SO17 1BJ, United Kingdom}
\email{S.D.Theriault@soton.ac.uk}



\begin{abstract} 
Beben and Wu showed that if $M$ is a $(2n-2)$-connected $(4n-1)$-dimensional 
Poincar\'{e} Duality complex such that $n\geq 3$ and $H^{2n}(M;\mathbb{Z})$ consists 
only of odd torsion, then $\Omega M$ can be decomposed up to homotopy as a product 
of simpler, well studied spaces. We use a result from~\cite{BT2} to greatly simplify 
and enhance Beben and Wu's work and to extend it in various directions. 
\end{abstract}

\maketitle

\section{Introduction} 

An orientable Poincar\'{e} Duality complex is a connected $CW$-complex whose 
cohomology satisfies Poincar\'{e} Duality. An orientable manifold is an example. 
In~\cite{BW} Beben and Wu gave a homotopy decomposition of $\Omega M$ where 
$M$ is any $(2n-2)$-connected $(4n-1)$-dimensional orientable Poincar\'{e} Duality complex, 
provided $n\geq 3$ and $H^{2n}(M;\mathbb{Z})$ has no $2$-torsion. They used 
this to show that the homotopy type of $\Omega M$ depended only on homological 
properties of $M$. This is in contrast to the homotopy type of $M$, which is known 
to depend on other properties as well. In particular, their result implies that the homotopy 
groups of $M$ depend only on its homological properties. 

In this paper we revisit Beben and Wu's result. We give a simpler approach involving 
much less spectral sequence calculation, instead relying on a result proved in~\cite{BT2}. 
This allows for the results to be significantly extended and enhanced in various directions. 

It should also be noted that earlier work of Selick~\cite{Se} using different methods 
can be used to give a $p$-local homotopy decomposition of $\Omega M$ when $p$ is 
an odd prime and $H^{2n}(M;\mathbb{Z})\cong\mathbb{Z}/p^{r}\mathbb{Z}$. This has 
the advantage that it avoids calculating the mod-$p$ homology of $\Omega M$ entirely 
but it also cedes a level of precision that we will later require; this is explained more 
fully in Section~\ref{sec:rank1}. 
 
For any $(2n-2)$-connected $(4n-1)$-dimensional Poincar\'{e} Duality complex $M$ 
we have 
\[H^{2n}(M;\mathbb{Z})\cong\mathbb{Z}^{d}\oplus 
       \bigoplus_{k=1}^{\ell}\mathbb{Z}/p_{k}^{r_{k}}\mathbb{Z}\] 
where $d\geq 0$, each $p_{k}$ is prime and each $r_{k}\geq 1$. The case when 
$d\geq 1$ has been dealt with in~\cite[Examples 4.4 and 5.3]{BT2} (see~\cite{Bas} for a 
different approach to the homotopy type) so we restrict to the case when $d=0$. 
In this case the description of $H^{2n}(M;\mathbb{Z})$ implies that the $2n$-skeleton~$M_{2n}$ 
of $M$ is homotopy equivalent to a wedge of Moore spaces, 
$M_{2n}\simeq\bigvee_{k=1}^{\ell} P^{2n}(p_{k}^{r_{k}})$. If each~$p_{k}$ is 
an odd prime we show the following. Let $m$ be the least common multiple of 
$\{p_{1}^{r_{1}},\ldots,p_{\ell}^{r_{\ell}}\}$ and let $m=\bar{p}_{1}^{\bar{r}_{1}}\cdots\bar{p}_{s}^{\bar{r}_{s}}$ 
be its prime decomposition. Notice that $\{\bar{p}_{1},\ldots,\bar{p}_{s}\}$ is the set of distinct 
primes in $\{p_{1},\ldots,p_{\ell}\}$  and each~$\bar{r}_{j}$ is the maximum power of $\bar{p}_{j}$ 
appearing in the list $\{p_{1}^{r_{1}},\ldots,p_{\ell}^{r_{\ell}}\}$. By~\cite{N}, the wedge 
of Moore spaces $\bigvee_{j=1}^{s} P^{2n}(\bar{p}_{j}^{\bar{r}_{j}})$ is homotopy 
equivalent to $P^{2n}(m)$. Write $M_{2n}\simeq P^{2n}(m)\vee\Sigma A$ where 
$\Sigma A$ is the wedge of the remaining Moore spaces in $M_{2n}$. Let $f$ be the 
composite of inclusions 
\(f\colon\nameddright{\Sigma A}{}{M_{2n}}{}{M}\) 
and define the space~$V$ and the map $\frak{h}$ by the homotopy cofibration 
\(\nameddright{\Sigma A}{f}{M}{\frak{h}}{V}\). 
We show that $V$ is a Poincar\'{e} Duality complex with 
$H^{2n}(V;\mathbb{Z})\cong\mathbb{Z}/m\mathbb{Z}$, $\Omega\frak{h}$ has 
a right homotopy inverse 
\(s\colon\namedright{\Omega V}{}{\Omega M}\) 
and prove the following.  

In general, for a space $X$ let 
\(ev\colon\namedright{\Sigma\Omega X}{}{X}\) 
be the canonical evaluation map. Given two maps 
\(a\colon\namedright{\Sigma X}{}{Z}\) 
and 
\(b\colon\namedright{\Sigma Y}{}{Z}\), 
let 
\([a,b]\colon\namedright{\Sigma X\wedge Y}{}{Z}\) 
be the Whitehead product of $a$ and $b$. Let $S^{2n+1}\{p^{r}\}$ be 
the homotopy fibre of the degree~$p^{r}$ map on $S^{2n+1}$. 

\begin{theorem} 
   \label{introloopMdecomp} 
   Let $M$ be a $(2n-2)$-connected, $(4n-1)$-dimensional Poincar\'{e} Duality 
   complex such that $n\geq 2$. Suppose that 
   \[H^{2n}(M;\mathbb{Z})\cong\bigoplus_{k=1}^{\ell}\mathbb{Z}/p_{k}^{r_{k}}\mathbb{Z}\]  
   where each $p_{k}$ is an odd prime. Then with $V$ and $A$ chosen as above: 
   \begin{letterlist} 
      \item there is a homotopy fibration 
               \[\llnameddright{(\Sigma\Omega V\wedge A)\vee\Sigma A}{[\gamma,f]+f}{M}{\frak{h}}{V}\] 
               where $\gamma$ is the composite 
               \(\gamma\colon\nameddright{\Sigma\Omega V}{\Sigma s}{\Sigma\Omega M}{ev}{M}\); 
      \item the homotopy fibration in~(a) splits after looping to give a homotopy equivalence 
               \[\Omega M\simeq\Omega V\times\Omega((\Sigma\Omega V\wedge A)\vee\Sigma A);\] 
      \item there is a homotopy equivalence 
               \[\Omega V\simeq\prod_{j=1}^{s} S^{2n-1}\{\bar{p}_{j}^{\bar{r}_{j}}\}\times\Omega S^{4n-1}.\] 
   \end{letterlist} 
\end{theorem} 

As a notable special case, if the primes $p_{k}$ for $1\leq k\leq\ell$ all equal a common prime $p$, and 
$r$ is the maximum of $\{r_{1},\ldots,r_{k}\}$ then 
$\Omega V\simeq S^{2n-1}\{p^{r}\}\times\Omega S^{4n-1}$. 

In~\cite{BW} the decompositions in parts~(b) and~(c) of Theorem~\ref{introloopMdecomp} 
were proved for $n\geq 3$. Part~(a) is new as is the $n=2$ case for $2$-connected $7$-dimensional 
Poincar\'{e} Duality complexes. Further, while~\cite{BW} gives no information in $2$-torsion cases, 
in Theorem~\ref{genloopMdecomp} we prove analogues of parts~(a) and~(b) when 
$H^{2n}(M;\mathbb{Z})\cong\bigoplus_{k=1}^{\ell}\mathbb{Z}/p_{k}^{r_{k}}\mathbb{Z}\oplus\bigoplus_{s=1}^{t}\mathbb{Z}/2^{r_s}\mathbb{Z}$,  
where each $p_{k}$ is an odd prime, each $r_s\geq 2$, and $\ell\geq 1$, and in 
Proposition~\ref{p=2decomp} we consider special cases when $2$-primary analogues of 
part~(c) of Theorem~\ref{introloopMdecomp} hold.   

An interesting consequence is a rigidity result. In Remark~\ref{Wremark} 
we show that $(\Sigma\Omega V\wedge A)\vee\Sigma A$ is homotopy equivalent to a 
wedge $W$ of Moore spaces, so part~(b) may be written more succinctly as 
$\Omega M\simeq\Omega V\times\Omega W$. As observed in~\cite{BW}, 
the homotopy types of $\Omega V$ and $\Omega M$ depend only on 
information from $H^{2n}(M;\mathbb{Z})$. Thus, if $M$ and $M'$ are both $(2n-2)$-connected 
$(4n-1)$ dimensional Poincar\'{e} Duality complexes satisfying the hypotheses of 
Theorem~\ref{introloopMdecomp}, and $H^{2n}(M;\mathbb{Z})\cong H^{2n}(M';\mathbb{Z})$, 
then $\Omega M\simeq\Omega M'$. 

We also prove an additional statement that was unaddressed in~\cite{BW}. Let 
\(I\colon\namedright{M_{2n}}{}{M}\) 
be the inclusion of the $2n$-skeleton. We show that there is a homotopy cofibration 
\(\nameddright{P^{4n-1}(m)}{\mathfrak{G}}{M_{2n}\vee S^{4n-1}}{I+H}{M}\) 
where $\Omega(I+H)$ has a right homotopy inverse
\(S\colon\namedright{\Omega M}{}{\Omega(M_{2n}\vee S^{4n-1})}\), 
and prove the following. 

\begin{theorem} 
   \label{introI+Hfib} 
   With the same hypotheses as in Theorem~\ref{introloopMdecomp}, there is a homotopy fibration 
   \[\llnameddright{(P^{4n-1}(m)\wedge\Omega M)\vee P^{4n-1}(m)}{[\frak{G},\Gamma]+\frak{G}} 
          {M_{2n}\vee S^{4n-1}}{I+H}{M}\] 
   where $\Gamma$ is the composite 
   \(\nameddright{\Sigma\Omega M}{\Sigma S}{\Sigma\Omega(M_{2n}\vee S^{4n-1})}{ev}{M_{2n}\vee S^{4n-1}}\), 
   and this homotopy fibration splits after looping to give a homotopy equivalence 
   \[\Omega(M_{2n}\vee S^{4n-1})\simeq\Omega M\times 
          \Omega((P^{4n-1}(m)\wedge\Omega M)\vee P^{4n-1}(m)).\] 
\end{theorem} 

Theorem~\ref{introI+Hfib} is interesting. 
Since $M_{2n}$ is homotopy equivalent to a wedge of simply-connected Moore spaces, 
it is a suspension. The theorem therefore shows that $\Omega M$ retracts off a 
loop suspension, it identifies the complementary factor, and it explicitly describes how the 
complementary factor maps into the loop suspension. 

The authors would like to thank the referee for several comments that have improved the paper.

\section{Preliminary results} 
\label{sec:prelim} 

This section contains preliminary results that will be referred to frequently 
in the subsequent sections. We start with a general result from~\cite[Proposition 3.5]{BT2}. 

\begin{theorem} 
   \label{GTcofib} 
   Let 
   \(\nameddright{\Sigma A}{f}{Y}{h}{Z}\) 
   be a homotopy cofibration. Suppose that $\Omega h$ has a right 
   homotopy inverse 
   \(s\colon\namedright{\Omega Z}{}{\Omega Y}\). 
   Let $\gamma$ be the composite 
   \(\gamma\colon\nameddright{\Sigma\Omega Z}{\Sigma s}{\Sigma\Omega Y}{ev}{Y}\). 
   Then there there is a homotopy fibration 
   \[\llnameddright{(\Sigma\Omega Z\wedge A)\vee\Sigma A}{[\gamma,f]+f} 
           {Y}{h}{Z}\] 
   which splits after looping to give a homotopy equivalence 
   \[\hspace{5cm}\Omega Y\simeq\Omega Z\times\Omega((\Sigma\Omega Z\wedge A)\vee\Sigma A). 
         \hspace{5cm}\Box\] 
\end{theorem} 

\begin{remark} 
\label{GTcofibnat} 
As pointed out in~\cite[Remark 2.2]{T}, Theorem~\ref{GTcofib} has a naturality property. 
If there is a homotopy cofibration diagram 
\[\diagram 
     \Sigma A\rto^-{f}\dto & Y\rto^-{h}\dto & Z\dto \\ 
     \Sigma A'\rto^-{f'} & Y'\rto^-{h'} & Z' 
  \enddiagram\] 
and both $\Omega h$ and $\Omega h'$ have right homotopy inverses $s$ and $s'$ 
respectively such that there is a homotopy commutative diagram 
\[\diagram 
      \Omega Z\rto^-{s}\dto & \Omega Y\dto \\ 
      \Omega Z'\rto^-{s'} & \Omega Y' 
  \enddiagram\] 
then then the homotopy fibration in Theorem~\ref{GTcofib} is also natural. 
\end{remark}  

Next, we prove two general lemmas about the existence of certain right homotopy inverses. 

\begin{lemma} 
   \label{inverselemma} 
   Suppose that there is a homotopy equivalence 
   \[e\colon\nameddright{X\times Y}{f\times g}{\Omega Z\times\Omega Z}{\mu}{\Omega Z}\] 
   for some maps $f$ and $g$, where $\mu$ is the loop multiplication, and suppose that 
   there is a map 
   \(\namedright{\Omega W}{\Omega h}{\Omega Z}\). 
   If both $f$ and $g$ lift through $\Omega h$, then $\Omega h$ has a right homotopy inverse. 
\end{lemma} 

\begin{proof} 
Let 
\(s\colon\namedright{X}{}{\Omega W}\) 
and 
\(t\colon\namedright{Y}{}{\Omega W}\) 
be lifts of $f$ and $g$ respectively through $\Omega h$. Consider the diagram 
\[\diagram 
       X\times Y\rto^-{s\times t}\drto_{f\times g} 
             & \Omega W\times\Omega W\rto^-{\mu}\dto^{\Omega h\times\Omega h} 
             & \Omega W\dto^{\Omega h} \\ 
        & \Omega Z\times\Omega Z\rto^-{\mu} & \Omega Z. 
   \enddiagram\] 
The left triangle homotopy commutes by definition of $s$ and $t$ and the right square 
homotopy commutes since $\Omega h$ is an $H$-map. The lower direction around the 
diagram is the defintion of the homotopy equivalence $e$, so the upper row is a lift 
of $e$ through $\Omega W$. Therefore $\Omega h$ has a right homotopy inverse. 
\end{proof}

\begin{lemma} 
   \label{splittinglemma} 
   Suppose that there is a homotopy fibration diagram 
   \[\diagram 
         F\rto^-{f}\dto^{q} & F'\dto \\ 
         E\rto^-{e}\dto^{p} & E'\dto \\ 
         B\rto^-{b} & B' 
      \enddiagram\] 
   of path-connected $CW$-complexes, where $\Omega b$, $\Omega f$ and $\Omega p$ 
   have right homotopy inverses. Then $\Omega e$ has a right homotopy inverse. 
\end{lemma} 

\begin{proof} 
Let 
\(r\colon\namedright{\Omega B'}{}{\Omega B}\), 
\(s\colon\namedright{\Omega F'}{}{\Omega F}\)  
and 
\(t\colon\namedright{\Omega B}{}{\Omega E}\) 
be right homotopy inverses for $\Omega b$, $\Omega f$ and $\Omega p$ respectively. 
Let $\theta$ be the composite 
\[\theta\colon\lnamedddright{\Omega B'\times\Omega F'}{r\times s}{\Omega B\times\Omega F} 
        {t\times\Omega q}{\Omega E\times\Omega E}{\mu}{\Omega E}.\] 
Consider the diagram 
\[\diagram 
      \Omega F'\rto^-{s}\dto^{i_{2}} & \Omega F\rto^-{\Omega f}\dto^{\Omega q} 
           & \Omega F'\dto \\ 
      \Omega B'\times\Omega F'\rto^-{\theta}\dto^{\pi_{1}} 
           & \Omega E\rto^-{\Omega e}\dto^{\Omega p} & \Omega E'\dto \\ 
      \Omega B'\rto^-{r} & \Omega B\rto^-{\Omega b} & \Omega B' 
   \enddiagram\] 
where $i_{2}$ is the inclusion of the second factor and $\pi_{1}$ is the projection onto 
the first factor. The lower left square homotopy commutes by definition of $\theta$ and 
$\Omega p$ being an $H$-map. The left column is a fibration, so the homotopy commutativity 
of the lower left square implies there is an induced map of fibres 
\(\namedright{\Omega F'}{}{\Omega F}\). 
Since $\theta\circ i_{2}=\Omega q\circ s$, a choice of map of fibres is $s$. Thus the 
left side of the diagram is a map of homotopy fibrations, as is the right by hypothesis. 
Therefore the composite from the left to the right column is a self-map of a homotopy 
fibration in which the top and bottom maps are homotopic to the identity. Therefore, 
by the Five Lemma, $\Omega e\circ\theta$ induces an isomorphism on homotopy groups 
and so is a homotopy equivalence by Whitehead's Theorem. 
\end{proof}

\section{The case when $H^{2n}(M;\mathbb{Z})\cong\mathbb{Z}/p^{r}\mathbb{Z}$} 
\label{sec:rank1} 

Let $V$ be a $(2n-2)$-connected $(4n-1)$-dimensional Poincar\'{e} Duality complex 
with $H^{2n}(V;\mathbb{Z})\cong\mathbb{Z}/p^{r}\mathbb{Z}$ where $p$ is a prime. 
As a $CW$-complex, $V=P^{2n}(p^{r})\cup e^{4n-1}$, and there is a homotopy cofibration 
\[\nameddright{S^{4n-2}}{f}{P^{2n}(p^{r})}{i}{V}\] 
where $f$ is the attaching map for the top cell and $i$ is the inclusion of the $(4n-2)$-skeleton.  
In this section we prove Theorems~\ref{introloopMdecomp} and~\ref{introI+Hfib} in the 
special case when $M=V$, assuming that $n\geq 2$ and~$p$ is odd. Some $2$-primary 
cases of Theorem~\ref{introloopMdecomp}~(c) will be deferred to Section~\ref{sec:p=2}. 

A decomposition of $\Omega V$ was proved by Beben and Wu~\cite{BW} for $n\geq 3$ 
and $p$ odd. Their method was more elaborate as it kept track of homology information in the 
general case of $\Omega M$ where $M$ is any $(2n-2)$-connected $(4n-1)$-dimensional 
Poincar\'{e} Duality complex with degree~$2n$ cohomology consisting only of odd torsion. 
We give a much simpler approach to the general case in Section~\ref{sec:rankm}, 
and so only need to keep track of homology information for the special case of $V$. 

Different methods were used by Selick~\cite{Se} to give a $p$-local decomposition 
of $\Omega V$ for $n\geq 2$ and $p$ odd. He used a generalization of methods developed 
by Dyer-Lashof and Ganea to produce a $p$-lcoal homotopy fibration 
\[\nameddright{S^{2n-1}\{p^{r}\}}{\partial}{S^{4n-1}}{h'}{V}\] 
where $\partial$ is null homotopic, giving a $p$-local homotopy equivalence 
$\Omega V\simeq S^{2n-1}\{p^{r}\}\times\Omega S^{4n-1}$, analogous to our 
Proposition~\ref{loopVdecomp}. The advantage of Selick's method is that it avoids 
homology calculations entirely. Ideally, we would like this to be an integral result 
rather than a $p$-local one; the method itself cannot be upgraded to do this as it depends 
on $S^{2n-1}$ being an $H$-space which rarely happens integrally, however a 
Sullivan square type argument could potentially be used to rectify this given that~$V$ 
localized at primes not equal to $p$ or rationally is homotopy equivalent to $S^{4n-1}$. 
The real disadvantage of Selick's method for our purposes is that it does not describe 
the homotopy class of the map $h'$ with enough precision for later use 
in Lemma~\ref{i+hinv} and Proposition~\ref{i+hfib}. It would be interesting to see if 
his techniques could be enhanced to do this, but in the meantime we fall back to 
homology calculations to deal with $\Omega V$. 

\begin{lemma} 
   \label{deg4n} 
   In degrees~$\leq 4n$ there is an algebra isomorphism 
   $\hlgy{\Omega V;\mathbb{Z}/p\mathbb{Z}}\cong\mathbb{Z}/p\mathbb{Z}[x,y]$ 
   where $\vert x\vert=2n-2$ and $\vert y\vert=2n-1$. 
\end{lemma} 

\begin{proof} 
Throughout, take cohomology and homology with mod-$p$ coefficients. 
By Poincar\'{e} Duality, there is an algebra isomorphism 
$\cohlgy{V}\cong\Lambda(a,b)$ where $\vert a\vert=2n-1$, $\vert b\vert=2n$. Dualizing, 
there is a coalgebra isomorphism $\hlgy{V}\cong\Lambda(u,v)$ where $u,v$ are the duals 
of $a,b$ respectively. In particular, if $\overline{\Delta}$ is the reduced diagonal, then $u$ 
and $v$ are primitive and $\overline{\Delta}(uv)=u\otimes v+v\otimes u$. 

Consider the mod-$p$ homology Serre spectral sequence for the principal homotopy fibration 
\(\nameddright{\Omega V}{}{\ast}{}{V}\). 
We have $E^{2}\cong\hlgy{\Omega V}\otimes\hlgy{V}$ and the spectral sequence 
converges to $\hlgy{\ast}$. For degree reasons, the first possible nontrivial differential 
is $d^{2n-1}$, and for convergence reasons we must have $d^{2n-1}(u)=x$ for some 
$x\in H_{2n-2}(\Omega V)$. Also, for convergence reasons we must have $d^{2n}(v)=y$ 
for some $y\in H_{2n-1}(\Omega V)$. Thus, in the $E^{2}$-term, we also have the elements 
$x\otimes u, x\otimes v, y\otimes u,y\otimes v$. Since the spectral sequence is principal, 
$d^{2n-1}$ and $d^{2n}$ are differentials, so $d^{2n-1}(x\otimes u)=x^{2}$, $d^{2n-1}(y\otimes u)=xy$ 
and $d^{2n}(y\otimes v)=y^{2}$. We claim that $d^{2n-1}(uv)=t\cdot (x\otimes v)$ 
for some unit $t\in\mathbb{Z}/p\mathbb{Z}$. The diagonal map gives a morphism of fibrations from 
\(\nameddright{\Omega V}{}{\ast}{}{V}\) 
to 
\(\nameddright{\Omega V\times\Omega V}{}{\ast}{}{V\times V}\) 
that induces a morphism of mod-$p$ homology Serre spectral sequences. Note that 
in the product fibration there is a K\"{u}nneth isomorphism that lets us regard the homology 
of the product as the tensor product of the homologies of the factors. Since the diagonal 
map induces the coalgebra structure in homology, this morphism of Serre spectral sequences 
implies that the differentials commute with the reduced diagonal. Therefore  
$\overline{\Delta}(d^{2n-1}(uv))=(d^{2n-1}\otimes d^{2n-1})(u\otimes v+v\otimes u)=x\otimes v+v\otimes x$ 
(noting that $d^{2n-1}(v)=0$). In particular, $\overline{\Delta}(d^{2n-1}(uv))\neq 0$, 
so $d^{2n-1}(uv)\neq 0$. For degree reasons, this implies that $d^{2n-1}(uv)=t\cdot (x\otimes v)$ 
for some unit $t\in\mathbb{Z}/p\mathbb{Z}$. Thus at the $E^{2n+1}$-page all elements of 
degree~$\leq 4n$ have vanished. Consequently, in degrees~$\leq 4n$ there is an algebra 
isomorphism $\hlgy{\Omega V}\cong\mathbb{Z}/p\mathbb{Z}[x,y]$. 
\end{proof} 

Next consider the effect of the map 
\(\namedright{\Omega P^{2n}(p^{r})}{\Omega i}{\Omega V}\) 
in mod-$p$ homology. Since $n\geq 2$, $P^{2n}(p^{r})$ is a suspension, so by 
the Bott-Samelson Theorem there is an algebra isomorphism 
\[\hlgy{\Omega P^{2n}(p^{r});\mathbb{Z}/p\mathbb{Z}}\cong T(x,y)\] 
where $T(\ \ )$ is the free tensor algebra functor, $\vert x\vert=2n-2$, $\vert y\vert=2n-1$ 
and $\beta^{r}y=x$. Since $i$ is the inclusion of the $(4n-2)$-skeleton, it induces a 
homotopy equivalence in dimensions~$\leq 4n-3$, so $\Omega i$ induces a homotopy 
equivalence in dimensions~$\leq 4n-4$. In particular, $(\Omega i)_{\ast}$ induces an 
isomorphism in degrees $2n-2$ and $2n-1$ in mod-$p$ homology. As $(\Omega i)_{\ast}$ 
is an algebra map, from Lemma~\ref{deg4n} we obtain the following. 

\begin{lemma} 
   \label{leastker} 
   In mod-$p$ homology, the generator of least degree in the kernel of $(\Omega i)_{\ast}$ 
   is $[x,y]$.~$\qqed$ 
\end{lemma} 

Let 
\(\widetilde{f}\colon\namedright{S^{4n-3}}{}{\Omega P^{2n}(p^{r})}\) 
be the adjoint of $f$. Let $\iota_{m}\in H_{m}(S^{m};\mathbb{Z}/p\mathbb{Z})$ be 
a choice of a generator. 

\begin{lemma} 
   \label{tildefimage} 
   In mod-$p$ homology, there is a choice of $\iota_{4n-3}$ such that 
   $\widetilde{f}_{\ast}(\iota_{4n-3})=[x,y]$. 
\end{lemma} 

\begin{proof} 
Recall the cofibration 
\(\nameddright{S^{4n-2}}{f}{P^{2n}(p^{r})}{i}{V}\). 
Define the space $F$ by the homotopy fibration 
\(\nameddright{F}{}{P^{2n}(p^{r})}{i}{V}\) 
and consider the mod-$p$ homology Serre spectral sequence for the principal fibration 
\(\nameddright{\Omega V}{}{F}{}{P^{2n}(p^{r})}\). 
The $E^{2}$-page of the spectral sequence is given by $\hlgy{P^{2n}(p^{r})}\otimes\hlgy{\Omega V}$. 
Let $u,v$ be the generators of $\hlgy{P^{2n}(p^{r})}$ in degrees $2n-1,2n$ respectively. 
By Lemma~\ref{deg4n}, $\hlgy{\Omega V}\cong\mathbb{Z}/p\mathbb{Z}[x,y]$ in degrees~$\leq 4n$, 
where $\vert x\vert=2n-2$ and $\vert y\vert=2n-1$. Since~$i$ is the inclusion of the $2n$-skeleton, 
we have $d^{2n-1}(u)=x$ and $d^{2n}(v)=y$. As the fibration is principal, the differentials in 
the spectral sequence are derivations so we obtain $d^{2n-1}(u\otimes x)=x^{2}$, 
$d^{2n-1}(u\otimes y)=xy$ and $d^{2n}(v\otimes y)=v^{2}$. Thus by the $E^{2n+1}$-page 
of the spectral sequence there is only one element left in degrees~$\leq 4n-2$, and that is 
the image of the $E^{2}$-page element $v\otimes x$. For degree reasons, this element is 
in the kernel of all higher differentials and therefore survives the spectral sequence. Thus 
the $(4n-2)$-skeleton of $F$ is $S^{4n-2}$. 

Returning again to the cofibration 
\(\nameddright{S^{4n-2}}{f}{P^{2n}(p^{r})}{i}{V}\), 
there is clearly a lift 
\[\diagram 
       & S^{4n-2}\dto^{f}\dlto_{\lambda} \\ 
       F\rto & P^{2n}(p^{r}) 
  \enddiagram\] 
for some map $\lambda$. By the Blakers-Massey Theorem, $\lambda$ is a homotopy 
equivalence in dimensions less than $4n-2$, so up to multiplication by a unit, $\lambda$ may 
be regarded as the inclusion of the bottom cell of $F$. Taking adjoints, $\widetilde{f}$ 
factors as the composite 
\(\nameddright{S^{4n-3}}{\widetilde{\lambda}}{\Omega F}{}{\Omega P^{2n}(p^{r})}\), 
where $\widetilde{\lambda}$ is the adjoint of $\lambda$. Now $\widetilde{\lambda}$ 
is the inclusion of the bottom cell in $\Omega F$, and Lemma~\ref{leastker} implies that 
the inclusion of this bottom cell has image equal to the 
generator of least degree in the kernel of $(\Omega i)_{\ast}$, which is $[x,y]$. 
Thus there is a choice of generator $\iota_{4n-3}$ in $H_{4n-3}(S^{4n-3};\mathbb{Z}/p\mathbb{Z})$ 
such that $\widetilde{f}_{\ast}(\iota_{4n-3})=[x,y]$. 
\end{proof} 

The low degree calculations made so far now let us calculate $\hlgy{\Omega V;\mathbb{Z}/p\mathbb{Z}}$ 
and $(\Omega i)_{\ast}$ in full. 

\begin{proposition} 
   \label{hlgyloopV} 
   Let $V$ be a $(2n-2)$-connected, $(4n-1)$-dimensional Poincar\'{e} Duality 
   complex with $H_{2n-1}(V;\mathbb{Z})\cong\mathbb{Z}/p^{r}\mathbb{Z}$ 
   for $p$ a prime and $r\geq 1$. Then there is an algebra isomorphism 
   \[\hlgy{\Omega V;\mathbb{Z}/p\mathbb{Z}}\cong\mathbb{Z}/p\mathbb{Z}[x,y]\] 
   where $\vert x\vert=2n-2$, $\vert y\vert=2n-1$ and $\beta^{r}y=x$, where $\beta^{r}$ 
   is the $r^{th}$-Bockstein. Further,  in mod-$p$ homology the map 
   \(\namedright{\Omega P^{2n}(p^{r})}{\Omega i}{\Omega V}\) 
   induces the algebra epimorphism 
   \(\namedright{T(x,y)}{}{\mathbb{Z}/p\mathbb{Z}[x,y]}\).
\end{proposition} 

\begin{proof} 
In general, if $X$ is a simply-connected $CW$-complex and $R$ is a ring then there is an 
Adams-Hilton model $AH(X)$ for calculating $\hlgy{\Omega X;R}$ as an algebra. The 
model is a differential graded algebra of the form $T(a_{1},\ldots,a_{k};d)$ where 
$T(\ \ )$ is the free tensor algebra functor, there is a generator~$a_{i}$ for each cell 
of $X$, the degree of $a_{i}$ is one less than the dimension of the corresponding cell, 
and~$d$ is a differential. There is an algebra isomorphism $H(AH(X))\cong\hlgy{\Omega X; R}$. 

To describe $d$, let $X_{t}$ be the $t$-skeleton of $X$ and let 
\(\namedright{S^{t}}{f_{i}}{X_{t}}\) 
attach a $(t+1)$-cell corresponding to~$a_{i}$. Let $AH(X_{t})$ be the Adams-Hilton 
model obtained from $AH(X)$ by restriction to the generators corresponding to cells 
in $X_{t}$. Then $d(a_{i})$ is determined by the image of the adjoint 
\(\namedright{S^{t-1}}{\widetilde{f}_{i}}{\Omega X_{t}}\) 
in the Adams-Hilton model $AH(X_{t})$. 

In our case, as $V$ has three cells there is an Adams-Hilton model $AH(V)=T(x,y,z;d)$ 
with $\vert x\vert=2n-2$, $\vert y\vert=2n-1$ and $\vert z\vert=4n-2$, and an algebra 
isomorphism $H(AH(V))\cong\hlgy{\Omega V;\mathbb{Z}/p\mathbb{Z}}$. The inclusion 
of the $(4n-2)$-skeleton of $V$ is the map 
\(\namedright{P^{2n}(p^{r})}{i}{V}\), 
so $AH(P^{2n}(p^{r}))=T(x,y;d')$ is an Adams-Hilton model whose homology is 
isomorphic as an algebra to $\hlgy{\Omega P^{2n}(p^{r});\mathbb{Z}/p\mathbb{Z}}$. 
By the Bott-Samelson theorem, the latter is known to be $T(x,y)$, so $d'$ must be 
identically zero. Thus, in this case, 
$AH(P^{2n}(p^{r}))\cong\hlgy{\Omega P^{2n}(p^{r});\mathbb{Z}/p\mathbb{Z}}$, 
so to determine the differential $dz$, which corresponds to the attaching map 
\(\namedright{S^{4n-2}}{f}{P^{2n}(p^{r})}\) 
for the top cell of $V$, we need to determine the image in mod-$p$ homology of the adjoint 
\(\namedright{S^{4n-3}}{\tilde{f}}{\Omega P^{2n}(p^{r})}\). 
By Lemma~\ref{tildefimage} this image is $[x,y]$. Thus $dz=[x,y]$, so we obtain 
algebra isomorphisms 
\[\hlgy{\Omega V;\mathbb{Z}/p\mathbb{Z}}\cong H(AH(V))\cong 
    H(T(x,y,z; dz=[x,y]))\cong\mathbb{Z}/p\mathbb{Z}[x,y].\] 
Further, the skeletal inclusion 
\(\namedright{P^{2n}(p^{r})}{i}{V}\) 
induces the map of Adams-Hilton models 
\(\namedright{T(x,y;d')}{}{T(x,y,z;d)}\), 
which upon taking homology gives the projection 
\(\namedright{T(x,y)}{(\Omega i)_{\ast}}{\mathbb{Z}/p\mathbb{Z}[x,y]}\). 
\end{proof} 

Now specialize to $p$ being an odd prime; we will return to $p=2$ in Section~\ref{sec:p=2}. 
By~\cite{Bar}, for $m\leq (2n-2)p$ the homotopy groups $\pi_{m}(P^{2n}(p^{r}))$ 
have the property that $p^{r}\cdot\pi_{m}(P^{2n}(p^{r}))\cong 0$. Notice that 
$4n-2\leq (2n-2)p$ for all $n\geq 2$ and $p\geq 3$. Thus 
\(\namedright{S^{4n-2}}{f}{P^{2n}(p^{r})}\) 
extends to a map 
\(g\colon\namedright{P^{4n-1}(p^{r})}{}{P^{2n}(p^{r})}\), 
and there is a homotopy cofibration diagram 
\begin{equation} 
  \label{fgextend} 
  \diagram 
        S^{4n-2}\rto\ddouble & P^{4n-1}(p^{r})\rto^-{q}\dto^{g} & S^{4n-1}\dto^{h} \\ 
        S^{4n-2}\rto^-{f} & P^{2n}(p^{r})\rto^-{i} & V 
  \enddiagram 
\end{equation} 
where $q$ is the pinch map to the top cell and $h$ is an induced map of cofibres. Let 
\(\widetilde{h}\colon\namedright{S^{4n-2}}{}{\Omega V}\) 
be the adjoint of $h$. 

\begin{lemma} 
   \label{oddph} 
   Let $p$ be an odd prime and take mod-$p$ homology. If $n\geq 3$ then 
   $\widetilde{h}_{\ast}(\iota_{4n-2})=y^{2}$. If $n=2$ then 
   $\widetilde{h}_{\ast}(\iota_{4n-2})=y^{2}+t\cdot x^{3}$ for some $t\in\mathbb{Z}/p\mathbb{Z}$. 
\end{lemma} 

\begin{proof} 
Let 
\(\widetilde{g}\colon\namedright{P^{4n-2}(p^{r})}{}{\Omega P^{2n}(p^{r})}\) 
be the adjoint of $g$ and first consider $\widetilde{g}_{\ast}$. By Lemma~\ref{tildefimage}, 
in mod-$p$ homology we have $\widetilde{f}_{\ast}(\iota_{4n-3})=[x,y]$. So if $u$ and $v$ 
are the generators in dimensions $4n-3$ and $4n-2$ of $\hlgy{P^{4n-2}(p^{r});\mathbb{Z}/p\mathbb{Z}}$ 
respectively, then the left square in~(\ref{fgextend}) implies that $\widetilde{g}_{\ast}(u)=[x,y]$. 
The naturality of the Bockstein therefore implies that 
$[x,y]=\widetilde{g}_{\ast}(u)=\widetilde{g}_{\ast}(\beta^{r}(v))=\beta^{r}(\widetilde{g}_{\ast}(v))$. 
The only generator of $H_{4n-2}(\Omega P^{2n}(p^{r});\mathbb{Z}/p\mathbb{Z})$ 
with a nonzero $r^{th}$-Bockstein is $\beta^{r}(y^{2})=xy-yx= [x,y]$. Thus 
$\widetilde{g}_{\ast}(v)=y^{2}+z$ where 
$\beta^{r}(z)=0$. If $n\geq 3$ then, for degree reasons, the only generator of 
$H_{4n-2}(\Omega P^{2n}(p^{r});\mathbb{Z}/p\mathbb{Z})$ is $y^{2}$. Thus 
$\widetilde{g}_{\ast}(v)=y^{2}$. If $n=2$ then $H_{4n-2}(\Omega P^{2n}(p^{r});\mathbb{Z}/p\mathbb{Z})$ 
has one other generator, that being $x^{3}$, so $\widetilde{g}_{\ast}(v)=y^{2}+t\cdot x^{3}$ for 
some $t\in\mathbb{Z}/p\mathbb{Z}$. 

Next, consider $\widetilde{h}_{\ast}$. Note that $q$ is the suspension of the pinch map 
\(\namedright{P^{4n-2}(p^{r})}{\bar{q}}{S^{4n-2}}\). 
Taking adjoints for the right square in~(\ref{fgextend}) then 
implies that $\Omega i\circ\widetilde{g}\simeq\widetilde{h}\circ\bar{q}$. Since~$\bar{q}_{\ast}(v)$ 
is a choice of $\iota_{4n-2}$ and $(\Omega i)_{\ast}$ is an epimorphism by Proposition~\ref{hlgyloopV}, 
from the description of~$\widetilde{g}_{\ast}(v)$ we obtain 
$\widetilde{h}_{\ast}(\iota_{4n-2})=y^{2}$ if $n\geq 3$ and 
$\widetilde{h}_{\ast}(\iota_{4n-2})=y^{2}+t\cdot x^{3}$ for some $t\in\mathbb{Z}/p\mathbb{Z}$ 
if $n=2$. 
\end{proof} 

\begin{remark} 
\label{oddphremark} 
Observe that~(\ref{fgextend}) implies that in mod-$q$ homology for $q$ a prime different 
from $p$, or in rational homology, the map $h$ induces an isomorphism. 
\end{remark}  

For any prime $p$, let $S^{2n-1}\{p^{r}\}$ be the homotopy fibre of the degree~$p^{r}$ 
map on $S^{2n-1}$. The principal fibration 
\(\nameddright{\Omega S^{2n-1}}{}{S^{2n-1}\{p^{r}\}}{}{S^{2n-1}}\) 
is induced by the degree~$p^{r}$ map so the mod-$p$ homology Serre spectral sequence 
collapses at the $E^{2}$-term, giving an isomorphism of $\mathbb{Z}/p\mathbb{Z}$-modules 
\[\begin{split} 
     \hlgy{S^{2n-1}\{p^{r}\};\mathbb{Z}/p\mathbb{Z}} & \cong 
        \hlgy{S^{2n-1};\mathbb{Z}/p\mathbb{Z}}\otimes\hlgy{\Omega S^{2n-1};\mathbb{Z}/p\mathbb{Z}} \\ 
        & \cong\Lambda(a)\otimes\mathbb{Z}/p\mathbb{Z}[b] 
\end{split}\]  
where $\vert a\vert=2n-1$, $\vert b\vert=2n-2$ and $\beta^{r}(a)=b$. Since $P^{2n}(p^{r})$ 
is the homotopy cofibre of the degree~$p^{r}$ map on $S^{2n-1}$, there is a 
homotopy fibration diagram 
\begin{equation} 
  \label{curlydgrm} 
  \diagram 
      \Omega S^{2n-1}\rto\dto^{\Omega j} & S^{2n-1}\{p^{r}\}\rto\dto^{s} & S^{2n-1}\rto^-{p^{r}}\dto 
           & S^{2n-1}\dto^{j} \\ 
       \Omega P^{2n}(p^{r})\rdouble & \Omega P^{2n}(p^{r})\rto & \ast\rto & P^{2n}(p^{r}) 
  \enddiagram 
\end{equation}  
where $j$ is the inclusion of the bottom cell and $s$ is an induced map of fibres. 
Observe that $j$ is the suspension of the map 
\(\bar{j}\colon\namedright{S^{2n-2}}{}{P^{2n-1}(p^{r})}\) 
that includes the bottom cell, and this inclusion induces an isomorphism in degree $2n-2$ 
in mod-$p$ homology. The naturality of the Bott-Samelson Theorem therefore implies that 
$(\Omega j)_{\ast}=(\Omega\Sigma\bar{j})_{\ast}$ is an algebra map sending $\mathbb{Z}/p\mathbb{Z}[b]$ 
isomorphically onto the subalgebra $\mathbb{Z}/p\mathbb{Z}[x]\subseteq T(x,y)$.  
The left square in~(\ref{curlydgrm}) then implies that~$s_{\ast}$ sends 
$\mathbb{Z}/p\mathbb{Z}[b]\subseteq\hlgy{S^{2n-1}\{p^{r}\};\mathbb{Z}/p\mathbb{Z}}$ 
isomorphically onto the subalgebra $\mathbb{Z}/p\mathbb{Z}[x]\subseteq T(x,y)$. 
The $r^{th}$-Bockstein is a differential, implying that $s_{\ast}$ sends 
$\Lambda(a)\otimes\mathbb{Z}/p\mathbb{Z}[b]$ isomorphically onto the sub-module  
$\Lambda(y)\otimes\mathbb{Z}/p\mathbb{Z}[x]\subseteq T(x,y)$. 

Let $t$ be the composite 
\[t\colon\nameddright{S^{2n-1}\{p^{r}\}}{s}{\Omega P^{2n}(p^{r})}{\Omega i}{\Omega V}.\] 
Then the description of $s_{\ast}$ implies that $t_{\ast}$ is an injection onto the 
submodule $\Lambda(y)\otimes\mathbb{Z}/p\mathbb{Z}[x]\subseteq\mathbb{Z}/p\mathbb{Z}[x,y]$. 
Let $e$ be the composite 
\[e\colon\llnameddright{S^{2n-1}\{p^{r}\}\times\Omega S^{4n-1}}{t\times\Omega h} 
      {\Omega V\times\Omega V}{\mu}{\Omega V}\] 
where $\mu$ is the loop space multiplication. Again, we focus on odd primes, leaving 
$p=2$ to Section~\ref{sec:p=2}. 

\begin{proposition} 
   \label{loopVdecomp} 
   Let $p$ be an odd prime. If $n\geq 2$ then the map 
   \(\namedright{S^{2n-1}\{p^{r}\}\times\Omega S^{4n-1}}{e}{\Omega V}\) 
   is a homotopy equivalence. 
\end{proposition} 

\begin{proof} 
We will show that after localizing at each prime and rationally, $e$ is a homotopy 
equivalence. This would imply that $e$ is an integral homotopy equivalence. 

First consider the case when $n\geq 3$. Localizing at $p$,   
$\hlgy{\Omega S^{4n-1};\mathbb{Z}/p\mathbb{Z}}\cong\mathbb{Z}/p\mathbb{Z}[c]$ 
for $\vert c\vert=4n-2$. The restriction of $\Omega h$ to the bottom cell of $\Omega S^{4n-1}$ 
is $\widetilde{h}$, so by Lemma~\ref{oddph}, $(\Omega h)_{\ast}(c)=y^{2}$. As $(\Omega h)_{\ast}$ 
is an algebra map, it sends $\mathbb{Z}/p\mathbb{Z}[c]$ isomorphically onto the subalgebra 
$\mathbb{Z}/p\mathbb{Z}[y^{2}]\subseteq\mathbb{Z}/p\mathbb{Z}[x,y]$. The description 
of $t_{\ast}$ then implies that $e_{\ast}$ induces an isomorphism in mod-$p$ homology, 
implying that $e$ is a $p$-local homotopy equivalence by Whitehead's Theorem. 
Localized at a prime $q\neq p$ or rationally, $S^{2n-1}\{p^{r}\}$ is 
contractible, $V$ is equivalent to $S^{4n-1}$, and Remark~\ref{oddphremark} implies that 
$h$ is a $q$-local or rational homotopy equivalence. Thus, in these cases, $e$ is also 
a $q$-local or rational homotopy equivalence. 

Next, consider the case when $n=2$. Localize at $p$. Going back to the description of $V$ as a 
$CW$-complex, observe that the composite 
\(\nameddright{S^{6}}{f}{P^{4}(p^{r})}{q}{S^{4}}\) 
is null homotopic, where $q$ is the pinch map to the top cell. This is because 
the generator of $\pi_{6}(S^{4})\cong\mathbb{Z}/2\mathbb{Z}$ cannot factor through 
an odd primary Moore space. Thus $q$ extends to a map 
\(\overline{q}\colon\namedright{V}{}{S^{4}}\). 
Since $\overline{q}$ extends $q$, in mod-$p$ homology we have $\Omega\overline{q}$ inducing 
the projection  
\(\namedright{\mathbb{Z}/p\mathbb{Z}[x,y]}{}{\mathbb{Z}/p\mathbb{Z}[y]}\). 
Now consider the composite 
\(\nameddright{\Omega S^{7}}{\Omega h}{\Omega V}{\Omega\overline{q}}{\Omega S^{4}}\). 
The restriction of $\Omega h$ to the bottom cell is $\widetilde{h}$, so Lemma~\ref{oddph} 
implies that $(\Omega\overline{q}\circ\Omega h)_{\ast}$ is an injection onto the subalgebra 
$\mathbb{Z}/p\mathbb{Z}[y^{2}]\subseteq\mathbb{Z}/p\mathbb{Z}[y]$. Since 
$\Omega S^{4}\simeq S^{3}\times\Omega S^{7}$ because of the existence of an element 
of Hopf invariant one, there is a projection 
\(\pi\colon\namedright{\Omega S^{4}}{}{\Omega S^{7}}\) 
which in mod-$p$ homology projects $\mathbb{Z}/p\mathbb{Z}[y]$ onto $\mathbb{Z}/p\mathbb{Z}[y^{2}]$. 
Thus the composition 
\(\namedddright{\Omega S^{7}}{\Omega h}{\Omega V}{\Omega\overline{q}}{\Omega S^{4}} 
      {\pi}{\Omega S^{7}}\) 
induces an isomorphism in homology and so is a homotopy equivalence. Consequently, there 
is a homotopy equivalence $\Omega V\simeq F\times\Omega S^{7}$ where $F$ is the homotopy 
fibre of $\pi\circ\Omega\overline{q}$. Notice that as $\overline{q}$ extends $q$, from the definition 
of $s$ there is a homotopy commutative diagram 
\[\diagram 
       S^{3}\{p^{r}\}\rto^-{s}\dto & \Omega P^{4}(p^{r})\rto^-{\Omega i}\dto^{\Omega q} 
            & \Omega V\dto^{\Omega\overline{q}} \\ 
        S^{3}\rto^-{E} & \Omega S^{4}\rdouble & \Omega S^{4} 
  \enddiagram\] 
where $E$ is the inclusion of the bottom cell. The top row is the definition of $t$. 
Consequently, $\pi\circ\Omega\overline{q}\circ t$ is null homotopic, so $t$ lifts to a map 
\(\overline{t}\colon\namedright{S^{3}\{p^{r}\}}{}{F}\). 
Since $t_{\ast}$ is an injection in mod-$p$ homology, so is $\overline{t}_{\ast}$. The decomposition 
$\Omega V\simeq F\times\Omega S^{7}$ implies that $F$ has the same Euler-Poincar\'{e} 
series as $S^{3}\{p^{r}\}$, therefore $\overline{t}_{\ast}$ is an isomorphism. Hence the 
map $e$ induces an isomorphism in mod-$p$ homology and so is a $p$-local homotopy equivalence by Whitehead's Theorem. Localizing at a prime $q\neq p$ or rationally, 
arguing exactly as in the $n\geq 3$ case shows that $e$ is also a $q$-local or rational 
homotopy equivalence. 
\end{proof} 

We can go further. In general, suppose that there is a homotopy pushout 
\[\diagram 
       A\rto^-{a}\dto^{b} & B\dto^{c} \\ 
       C\rto^-{d} & D 
  \enddiagram\] 
of simply-connected spaces where $A$ is a suspension. The suspension hypothesis 
implies that the set of homotopy classes of maps $[A,Z]$ is a group for any space $Z$. 
A Mayer-Vietoris style argument then shows that there is a homotopy cofibration 
\[\lnameddright{A}{a-b}{B\vee C}{c+d}{D}.\] 
Since $P^{4n-1}(p^{r})$ is the suspension of $P^{4n-2}(p^{r})$, applying this to 
the right square in~(\ref{fgextend}) we obtain a homotopy cofibration 
\begin{equation} 
   \label{gVcofib} 
   \lnameddright{P^{4n-1}(p^{r})}{g-q}{P^{2n}(p^{r})\vee S^{4n-1}}{i+h}{V}.  
\end{equation} 
   
\begin{lemma} 
   \label{i+hinv} 
   The map 
   \(\llnamedright{\Omega(P^{2n}(p^{r})\vee S^{4n-1})}{\Omega(i+h)}{\Omega V}\) 
   has a right homotopy inverse. 
\end{lemma} 

\begin{proof} 
By Proposition~\ref{loopVdecomp} there is a homotopy equivalence 
\(\lnameddright{S^{2n-1}\{p^{r}\}\times\Omega S^{4n-1}}{t\times\Omega h} 
      {\Omega V\times\Omega V}{\mu}{\Omega V}\). 
By Lemma~\ref{inverselemma}, to show that $\Omega(i+h)$ has a right homotopy 
inverse it suffices to show that both $t$ and $\Omega h$ lift through $\Omega(i+h)$. 

Let 
\(i_{1}\colon\namedright{P^{2n}(p^{r})}{}{P^{2n}(p^{r})\vee S^{4n-1}}\) 
and 
\(i_{2}\colon\namedright{S^{4n-1}}{}{P^{2n}(p^{r})\vee S^{4n-1}}\) 
be the inclusions of the left and right wedge summands respectively. Then 
$(i+h)\circ i_{1}=i$ and $(i+h)\circ i_{2}=h$. By definition, $t=\Omega i\circ s$, 
so the composite 
\(\nameddright{S^{2n-1}\{p^{r}\}}{s}{\Omega P^{2n}(p^{r})}{\Omega i_{1}} 
      {\Omega(P^{2n}(p^{r})\vee S^{4n-1})}\stackrel{\Omega(i+h)}{\llarrow}\Omega V\) 
equals $t$, while 
\(\namedright{\Omega S^{4n-1}}{\Omega i_{2}} 
      {\Omega(P^{2n}(p^{r})\vee S^{4n-1})}\stackrel{\Omega(i+h)}{\llarrow}\Omega V\) 
is $\Omega h$. Thus both $t$ and $\Omega h$ lift through $\Omega(i+h)$, as required. 
\end{proof} 


Next, the homotopy fibre of $\Omega(i+h)$ is identified. Let 
\(s\colon\namedright{\Omega V}{}{\Omega(P^{2n}(p^{r})\vee S^{4n-1})}\) 
be a right homotopy inverse for $\Omega(i+h)$. Let $\gamma$ be the composite 
\[\gamma\colon\nameddright{\Sigma\Omega V}{\Sigma s} 
     {\Sigma\Omega(P^{2n}(p^{r})\vee S^{4n-1})}{ev}{P^{2n}(p^{r})\vee S^{4n-1}}.\] 
Let $\frak{g}=g-q$. 

\begin{proposition} 
   \label{i+hfib} 
   There is a homotopy fibration 
   \[\llnameddright{(P^{4n-1}(p^{r})\wedge\Omega V)\vee P^{4n-1}(p^{r})}{[\frak{g},\gamma]+\frak{g}} 
          {P^{2n}(p^{r})\vee S^{4n-1}}{i+h}{V}\] 
   which splits after looping to give a homotopy equivalence 
   \[\Omega(P^{2n}(p^{r})\vee S^{4n-1})\simeq\Omega V\times 
          \Omega((P^{4n-1}(p^{r})\wedge\Omega V)\vee P^{4n-1}(p^{r})).\] 
\end{proposition} 

\begin{proof} 
Since there is a homotopy cofibration 
\(\nameddright{P^{4n-1}(p^{r})}{\frak{g}}{P^{2n}(p^{r})\vee S^{4n-1}}{i+h}{V}\) 
and, by Lemma~\ref{i+hinv}, $\Omega(i+h)$ has a right homotopy inverse, the 
assertions follow immediately from Theorem~\ref{GTcofib}. 
\end{proof} 

Note that Proposition~\ref{loopVdecomp} proves Theorem~\ref{introloopMdecomp} in the special 
case when $M=V$ while Proposition~\ref{i+hfib} proves Theorems~\ref{introI+Hfib}.

\section{The general case when $H^{2n}(M;\mathbb{Z})$ is odd torsion} 
\label{sec:rankm} 

Let $M$ be a $(2n-2)$-connected $(4n-1)$-dimensional Poincar\'{e} Duality 
complex such that $n\geq 2$ and 
\[H^{2n}(M;\mathbb{Z})\cong\bigoplus_{k=1}^{\ell}\mathbb{Z}/p_{k}^{r_{k}}\mathbb{Z}\]  
where each $p_{k}$ is an odd prime. Then the $2n$-skeleton $M_{2n}$ of $M$ is homotopy 
equivalent to a wedge of Moore spaces  
\[M_{2n}\simeq\bigvee_{k=1}^{\ell} P^{2n}(p_{k}^{r_{k}}).\] 
For $1\leq k\leq\ell$, let $a_{k}\in H^{2n-1}(M;\mathbb{Z}/p_{k}\mathbb{Z})$ 
and $b_{k}\in H^{2n}(M;\mathbb{Z}/p_{k}\mathbb{Z})$ be generators corresponding 
to the wedge summand $P^{2n}(p_{k}^{r_{k}})$ of $M_{2n}$. In~\cite[Section 6]{BW}, 
Beben and Wu used a Poincar\'{e} Duality argument to prove the following. 

\begin{lemma} 
   \label{pkchoice} 
   Let $p\in\{p_{1},\ldots,p_{\ell}\}$ be an odd prime. Let $\{i_{1},\ldots,i_{t}\}\subseteq\{1,\ldots,\ell\}$ 
   be the subset satisfying $p_{i_{j}}=p$ and let $r=\max\{r_{i_{1}},\ldots,r_{i_{t}}\}$. If 
   $p_{i_{j}}^{r_{i_{j}}}=p^{r}$ then $a_{i_{j}}\cup b_{i_{j}}$ is a generator of 
   $H^{4n-1}(M;\mathbb{Z}/p\mathbb{Z})$.~$\qqed$ 
\end{lemma} 

As in the Introduction, let $m$ be the least common multiple of 
$\{p_{1}^{r_{1}},\ldots,p_{\ell}^{r_{\ell}}\}$ and let $m=\bar{p}_{1}^{\bar{r}_{1}}\cdots\bar{p}_{s}^{\bar{r}_{s}}$ 
be its prime decomposition. Notice that $\{\bar{p}_{1},\ldots,\bar{p}_{s}\}$ is the set of distinct 
primes in $\{p_{1},\ldots,p_{\ell}\}$  and each~$\bar{r}_{j}$ is the maximum power of $\bar{p}_{j}$ 
appearing in the list $\{p_{1}^{r_{1}},\ldots,p_{\ell}^{r_{\ell}}\}$. In general, if $a$ and $b$ 
are coprime then by~\cite[proof of Proposition 1.5]{N} there is a homotopy equivalence 
$P^{t}(ab)\simeq P^{t}(a)\vee P^{t}(b)$. In our case, since $\{\bar{p}_{1},\ldots,\bar{p}_{s}\}$ 
are distinct primes and $m=\bar{p}_{1}^{\bar{r}_{1}}\cdots\bar{p}_{s}^{\bar{r}_{s}}$, 
there is a homotopy equivalence 
\[P^{2n}(m)\simeq\bigvee_{j=1}^{s} P^{2n}(\bar{p}_{j}^{\bar{r}_{j}}).\]  
Therefore $M_{2n}$ can be rewritten as 
\begin{equation} 
  \label{M2n2} 
  M_{2n}\simeq P^{2n}(m)\vee\Sigma A  
\end{equation} 
where $\Sigma A$ is the wedge of the remaining Moore spaces in $M_{2n}$. 

Define $j'$ and $j$ by the composites 
\[j'\colon P^{2n}(m)\hookrightarrow\nameddright{P^{2n}(m)\vee\Sigma A}{\simeq}{M_{2n}}{}{M}\]  
\[j\colon \Sigma A\hookrightarrow\nameddright{P^{2n}(m)\vee\Sigma A}{\simeq}{M_{2n}}{}{M}.\] 
Define the space $V$ and the map $\frak{h}$ by the homotopy cofibration 
\[\nameddright{\Sigma A}{j}{M}{\frak{h}}{V}.\] 
Then $V$ is a three-cell complex, $V=P^{2n}(m)\cup e^{4n-1}$, and the inclusion 
of the $(4n-2)$-skeleton is given by the composite 
\[i\colon\nameddright{P^{2n}(m)}{j'}{M}{\frak{h}}{V}.\] 
Observe that Lemma~\ref{pkchoice} implies that $V$ is a Poincar\'{e} Duality complex 
since the power of each $\bar{p}_{j}$ appearing as a factor of 
$m=\bar{p}_{1}^{\bar{r}_{1}}\cdots\bar{p}_{s}^{\bar{r}_{s}}$ is maximal. 

Let 
\(F\colon\namedright{S^{4n-2}}{}{M_{2n}}\) 
be the attaching map for the top cell of $M$. Define $f$ by the composite 
\(f\colon\nameddright{S^{4n-2}}{F}{M_{2n}}{\frak{q}}{P^{2n}(m)}\) 
where $\frak{q}$ collapses $\Sigma A$ in $M_{2n}\simeq P^{2n}(m)\vee\Sigma A$ to a point. 
Observe that there is a homotopy pushout diagram 
\begin{equation} 
  \label{Fpo} 
  \diagram 
     & \Sigma A\rdouble\dto & \Sigma A\dto^{j} \\ 
     S^{4n-2}\rto^-{F}\ddouble & M_{2n}\rto\dto^{\frak{q}} & M\dto^{\frak{h}} \\ 
     S^{4n-2}\rto^-{f} & P^{2n}(m)\rto^-{i'} & V 
  \enddiagram 
\end{equation}  
that defines the map $i'$. By definition of $\frak{q}$ the composite 
\(P^{2n}(m)\hookrightarrow\nameddright{P^{2n}(m)\vee\Sigma A}{\simeq}{M_{2n}}{\frak{q}}{P^{2n}(m)}\) 
is the identity map. Therefore $i'$ is homotopic to the composite 
\(P^{2n}(m)\hookrightarrow\namedddright{P^{2n}(m)\vee\Sigma A}{\simeq}{M_{2n}}{}{M}{\frak{h}}{V}\), 
which, by definition of $j'$, is $\frak{h}\circ j'$. But $\frak{h}\circ j'$ is the definition of $i$, 
so we have $i'=i$. Therefore $f$ is the attaching map for the top cell of $V$, and $F$ is a 
lift of $f$ through $\frak{q}$. 

We wish to show that $\Omega\frak{h}$ has a right homotopy inverse. Doing so will 
involve decomposing $\Omega V$ in a manner analogous to that for the special case 
when $m=p^{r}$ in Section~\ref{sec:rank1}. We first aim for the analogue of~(\ref{fgextend}). 

\begin{lemma} 
   \label{Fextend} 
   The map 
   \(\namedright{S^{4n-2}}{F}{M_{2n}}\) 
   extends to a map 
   \(G\colon\namedright{P^{4n-1}(m)}{}{M_{2n}}\). 
\end{lemma} 

\begin{proof} 
It is equivalent to show that the map $F$ has order~$m$, and showing this is 
equivalent to showing that the adjoint 
\(\widetilde{F}\colon\namedright{S^{4n-3}}{}{\Omega M_{2n}}\) 
of $F$ has order~$m$. 

Since $M_{2n}\simeq\bigvee_{k=1}^{\ell} P^{2n}(p_{k}^{r_{k}})$, by the Hilton-Milnor Theorem 
\[\Omega M_{2n}\simeq\prod_{k=1}^{\ell}\Omega P^{2n}(p_{k}^{r_{k}})\times 
     \prod_{j=1}^{\binom{\ell}{2}}\Omega(\Sigma P^{2n-1}(p_{j_{1}}^{r_{j_{1}}})\wedge P^{2n-1}(p_{j_{2}}^{r_{j_{2}}})) 
     \times N\] 
where $1\leq j_{1},j_{2}\leq\ell$, $j_{1}\neq j_{2}$, and $N$ is $(4n-2)$-connected. 
Thus $\widetilde{F}$ is a sum of maps of the form 
\(\widetilde{F}_{k}\colon\namedright{S^{4n-3}}{}{\Omega P^{2n}(p_{k}^{r_{k}})}\) 
and 
\(\widetilde{F}_{j}\colon\namedright{S^{4n-3}}{} 
     {\Omega(\Sigma P^{2n-1}(p_{j_{1}}^{r_{j_{1}}})\wedge P^{2n-1}(p_{j_{2}}^{r_{j_{2}}}))}\). 
As before, by~\cite{Bar} each map $\widetilde{F}_{k}$ 
has order at most $p_{k}^{r_{k}}$. As $p_{k}^{r_{k}}$ is a factor of $m$, we obtain a null 
homotopy for $\widetilde{F}_{k}\circ m$, for $1\leq k\leq\ell$. By~\cite[Corollary 6.6]{N}, if $p$ and $q$ 
are distinct primes then $P^{a}(p^{r})\wedge P^{b}(q^{s})$ is contractible, and if $r\leq s$ 
and $p^{r}\neq 2$, then $P^{a}(p^{r})\wedge P^{b}(p^{s})\simeq P^{a+b}(p^{r})\vee P^{a+b-1}(p^{r})$. 
Thus if $p_{j_{1}}\neq p_{j_{2}}$ then $\widetilde{F}_{j}$ is null homotopic, while if 
$p_{j_{1}}=p_{j_{2}}$ and we assume without loss of generality that $r_{j_{1}}\leq r_{j_{2}}$, 
then for dimensional reasons the Hilton-Milnor Theorem implies that $\widetilde{F}_{j}$ factors through 
\(\widehat{F}_{j}\colon\namedright{S^{4n-3}}{} 
      {\Omega P^{4n-1}(p_{j_{1}}^{r_{j_{1}}})\times\Omega P^{4n-2}(p_{j_{1}}^{r_{j_{1}}})}\). 
 For dimension and connectivity reasons, $\widehat{F}_{j}$ is trivial on the 
 $\Omega P^{4n-1}(p_{j_{1}}^{r_{j_{1}}})$ factor and is a multiple of the inclusion of the bottom 
 cell on the $\Omega P^{4n-2}(p_{j_{1}}^{r_{j_{1}}})$ factor. This inclusion has order~$p_{j_{1}}^{r_{j_{1}}}$, 
 so as $p_{j_{1}}^{r_{j_{1}}}$ is a factor of $m$, we obtain a null homotopy for 
$\widehat{F}_{j}\circ m$, and therefore one for $\widetilde{F}\circ m$. Hence 
$F\circ m$ is null homotopic. 
\end{proof} 

Lemma~\ref{Fextend} implies that there is a homotopy cofibration diagram 
\begin{equation} 
  \label{FGextend} 
  \diagram 
        S^{4n-2}\rto\ddouble & P^{4n-1}(m)\rto^-{q}\dto^{G} & S^{4n-1}\dto^{H} \\ 
        S^{4n-2}\rto^-{F} & M_{2n}\rto^-{I} & M 
  \enddiagram 
\end{equation} 
where $I$ is the skeletal inclusion, $q$ is the pinch map to the top cell, and $H$ is an induced 
map of cofibres. Combining this with~(\ref{Fpo}) gives an iterated homotopy pushout diagram 
\begin{equation} 
  \label{FGcombo} 
  \diagram 
        S^{4n-2}\rto\ddouble & P^{4n-1}(m)\rto^-{q}\dto^{G} & S^{4n-1}\dto^{H} \\ 
        S^{4n-2}\rto^-{F}\ddouble & M_{2n}\rto^-{I}\dto^{\mathfrak{q}} & M\dto^{\frak{h}} \\ 
        S^{4n-2}\rto^-{f} & P^{2n}(m)\rto^-{i} & V.  
  \enddiagram 
\end{equation} 

We now give a homotopy decomposition of $\Omega V$. By definition, 
$P^{2n}(m)\simeq\bigvee_{j=1}^{s} P^{2n}(\bar{p}_{j}^{\bar{r}_{j}})$. For $1\leq j\leq s$, 
define $S_{j}$ by the composite 
\[S_{j}\colon\nameddright{S^{2n-1}\{\bar{p}_{j}^{\bar{r}_{j}}\}}{s_{j}}{\Omega P^{2n}(\bar{p}_{j}^{\bar{r}_{j}})} 
      {\Omega i_{j}}{\Omega P^{2n}(m)}\] 
where $s_{j}$ is from~(\ref{curlydgrm}) and $i_{j}$ is the inclusion of the 
$j^{th}$-wedge summand. Define $S$ by the composite 
\[S\colon\llnameddright{\prod_{j=1}^{s} S^{2n-1}\{\bar{p}_{j}^{\bar{r}_{j}}\}}{\prod_{j=1}^{s} S_{j}} 
     {\prod_{j=1}^{s}\Omega P^{2n}(m)}{\mu}{\Omega P^{2n}(m)}\] 
and define $T$ by the composite 
\[T\colon\nameddright{\prod_{j=1}^{s} S^{2n-1}\{\bar{p}_{j}^{\bar{r}_{j}}\}}{S}{\Omega P^{2n}(m)} 
      {\Omega i}{\Omega V}.\] 
Finally, define $e$ by the composite 
\[e\colon\lllnameddright{\bigg(\prod_{j=1}^{s} S^{2n-1}\{\bar{p}_{j}^{\bar{r}_{j}}\}\bigg)\times\Omega S^{4n-1}} 
      {T\times\Omega(\mathfrak{h}\circ H)}{\Omega V\times\Omega V}{\mu}{\Omega V}.\]  

\begin{proposition} 
   \label{Vmdecomp} 
   If $n\geq 2$ then the map $e$ is a homotopy equivalence. 
\end{proposition} 

\begin{proof} 
We will show that after localizing at each prime $p$ and rationally, $e$ is a homotopy 
equivalence. This would imply that $e$ is an integral homotopy equivalence. 

Localize at a prime $p$ where $p=\bar{p}_{j}$ for some $1\leq j\leq s$. Let $r=\bar{r}_{j}$. 
If $q$ is a prime distinct from~$p$ then the Moore space $P^{a}(q^{s})$ 
is contractible for $a\geq 2$. Therefore, as $P^{a}(m)\simeq\bigvee_{j=1}^{s} P^{a}(\bar{p}_{j}^{\bar{r}_{j}})$  
and the primes $\bar{p}_{1},\ldots,\bar{p}_{s}$ are distinct, there is a $p$-local homotopy equivalence. 
\[P^{a}(m)\simeq P^{a}(p^{r}).\] 
Applying this to~(\ref{FGcombo}) we obtain a $p$-local homotopy cofibration diagram 
\[\diagram 
        S^{4n-2}\rto\ddouble & P^{4n-1}(p^{r})\rto^-{q}\dto^{g} & S^{4n-1}\dto^{h} \\ 
        S^{4n-2}\rto^-{f} & P^{2n}(p^{r})\rto^-{i} & V 
  \enddiagram\] 
where $g=\mathfrak{q}\circ G$ and $h=\mathfrak{h}\circ H$. This is a $p$-local version 
of~(\ref{fgextend}) so we may argue as in Lemma~\ref{oddph} and Proposition~\ref{loopVdecomp} 
to show that the composite 
\[S^{2n-1}\{p^{r}\}\times\Omega S^{4n-1}\hookrightarrow 
     \namedright{\bigg(\prod_{j=1}^{s} S^{2n-1}\{\bar{p}_{j}^{\bar{r}_{j}}\}\bigg)\times\Omega S^{4n-1}} 
      {e}{\Omega V}\] 
is a $p$-local homotopy equivalence. Notice that the spaces $S^{2n-1}\{\bar{p}_{j}^{\bar{r}_{j}}\}$ 
are contractible if $\bar{p}_{j}\neq p$, so in fact we have shown that $e$ is a $p$-local homotopy 
equivalence. 

Next, localize at a prime $p\notin\{\bar{p}_{1},\ldots,\bar{p}_{s}\}$. Then $P^{a}(m)$ 
for $a\geq 2$ and the Moore space wedge summands of $M_{2n}$ are all contractible. 
Therefore in~(\ref{FGcombo}) both $M$ and $V$ are homotopy equivalent to $S^{4n-1}$ and 
the maps $H$ and $\mathfrak{h}$ are both homotopy equivalences. On the other hand, 
the spaces $S^{2n-1}\{\bar{p}_{j}^{\bar{r}_{j}}\}$ are also contractible so $e$ reduces to 
$\Omega(\mathfrak{h}\circ H)$, which we have just seen is a homotopy equivalence. 
The same argument shows that $e$ is also a rational homotopy equivalence. 
\end{proof} 

Proposition~\ref{Vmdecomp} will be used to show that the map 
\(\namedright{\Omega M}{\Omega\frak{h}}{\Omega V}\) 
has a right homotopy inverse. Thinking ahead, this is drawn from a slightly stronger statement. 

\begin{lemma} 
   \label{preMhinv} 
   The composite 
   \(\llnameddright{\Omega(P^{2n}(m)\vee S^{4n-1})}{\Omega(j'+H)}{\Omega M}{\Omega\frak{h}}{\Omega V}\) 
   has a right homotopy inverse. 
\end{lemma} 

\begin{proof} 
By Proposition~\ref{Vmdecomp} there is a homotopy equivalence 
\[\lllnameddright{\bigg(\prod_{j=1}^{s} S^{2n-1}\{\bar{p}_{j}^{\bar{r}_{j}}\}\bigg)\times\Omega S^{4n-1}} 
      {T\times\Omega(\mathfrak{h}\circ H)}{\Omega V\times\Omega V}{\mu}{\Omega V}.\]  
By Lemma~\ref{inverselemma}, to show that $\Omega\mathfrak{h}\circ\Omega (j'+H)$ has 
a right homotopy inverse it suffices to show that both $T$ and $\Omega(\mathfrak{h}\circ H)$ 
lift through $\Omega\mathfrak{h}\circ\Omega(j'+H)$. 

By definition, $T$ is the composite 
\(\nameddright{\prod_{j=1}^{s} S^{2n-1}\{\bar{p}_{j}^{\bar{r}_{j}}\}}{S}{\Omega P^{2n}(m)} 
      {\Omega i}{\Omega V}\) 
and, by definition, $i$ is the composite 
\(\nameddright{P^{2n}(q)}{j'}{M}{\mathfrak{h}}{V}\).  
Thus $T=\Omega\mathfrak{h}\circ\Omega j'\circ S$. This implies that~$T$ 
lifts through $\Omega\mathfrak{h}\circ\Omega j'$ and hence through 
$\Omega\mathfrak{h}\circ\Omega(j'+H)$. Clearly, 
$\Omega(\mathfrak{h}\circ H)\simeq\Omega\mathfrak{h}\circ\Omega H$ lifts through 
$\Omega\mathfrak{h}\circ\Omega(j'+H)$. 
\end{proof} 

\begin{corollary} 
   \label{Mhinv} 
   The map 
   \(\namedright{\Omega M}{\Omega\frak{h}}{\Omega V}\) 
   has a right homotopy inverse.~$\qqed$  
\end{corollary} 

We can now prove Theorem~\ref{introloopMdecomp}. 

\begin{proof}[Proof of Theorem~\ref{introloopMdecomp}] 
From the homotopy cofibration 
\(\nameddright{\Sigma A}{j}{M}{\frak{h}}{V}\) 
and the right homotopy inverse of $\Omega\frak{h}$ in Corollary~\ref{Mhinv}, 
parts~(a) and~(b) follow immediately from Theorem~\ref{GTcofib}. Part~(c) is 
Proposition~\ref{Vmdecomp}.  
\end{proof} 

\begin{remark} 
\label{Wremark} 
By Theorem~\ref{introloopMdecomp}, 
$\Omega M\simeq\Omega V\times\Omega((\Sigma\Omega V\wedge A)\vee\Sigma A)$. 
We claim that $(\Sigma\Omega V\wedge A)\vee\Sigma A$ is homotopy equivalent to 
a wedge $W$ of spheres and odd primary Moore spaces. If so then we may more simply write 
$\Omega M\simeq\Omega V\times\Omega W$. To prove the claim, fist consider 
\[\Sigma\Omega V\simeq 
      \Sigma\bigg(\big(\prod_{j=1}^{s} S^{2n-1}\{\bar{p}_{j}^{\bar{r}_{j}}\}\big)\times\Omega S^{4n-1}\bigg).\] 
In general, if $B$ and $C$ are path-connected spaces then 
$\Sigma(B\times C)\simeq\Sigma B\vee\Sigma C\vee(\Sigma B\wedge C)$;  
by~\cite{J} the space $\Sigma\Omega S^{t+1}$ is homotopy equivalent to a wedge 
of suspended spheres; by~\cite{CMN} the space $\Sigma S^{2n-1}\{p^{r}\}$ is homotopy 
equivalent to a wedge of mod-$p^{r}$ Moore spaces; and by~\cite[Corollary 6.6]{N} there is a 
homotopy equivalence $P^{a}(p^{r})\wedge P^{b}(p^{s})\simeq P^{a+b}(p^{s})\vee P^{a+b-1}(p^{s})$ 
if $s\leq r$ and $p$ is odd while $P^{a}(p^{r})\wedge P^{b}(q^{s})$ is contractible if $p$ and $q$ 
are distinct primes. Collectively, these statements imply that $\Sigma\Omega V$ is homotopy 
equivalent to a wedge of spheres and odd primary Moore spaces. Since~$A$ is defined as a 
wedge of odd primary Moore spaces, we therefore also obtain that 
$(\Sigma\Omega V\wedge A)\vee\Sigma A$ is homotopy equivalent to a wedge of 
spheres and odd primary Moore spaces. 
\end{remark} 

Next,  we consider the analogue of Proposition~\ref{i+hfib}. This will be done 
in two steps, first with respect to 
\(\namedright{P^{2n}(m)}{i}{V}\) 
and then with respect to 
\(\namedright{M_{2n}}{I}{M}\). 
First, the homotopy pushout 
\[\diagram 
      P^{4n-1}(m)\rto^-{q}\dto^{\mathfrak{q}\circ G} & S^{4n-1}\dto^{\mathfrak{h}\circ H} \\ 
      P^{2n}(m)\rto^-{i} & V 
  \enddiagram\] 
in~(\ref{FGcombo}) implies that there is a homotopy cofibration 
\[\llnameddright{P^{4n-1}(m)}{(\mathfrak{q}\circ G)-q}{P^{2n}(m)\vee S^{4n-1}}{i+(\mathfrak{h}\circ H)}{V}.\] 

\begin{lemma} 
   \label{i+frakhinv} 
   The map 
   \(\lllnamedright{\Omega(P^{2n}(m)\vee S^{4n-1})}{\Omega(i+(\mathfrak{h}\circ H))}{\Omega V}\) 
   has a right homotopy inverse. 
\end{lemma} 

\begin{proof} 
This follows immediately from Lemma~\ref{preMhinv} since $i=\mathfrak{h}\circ j'$. 
\end{proof} 

Second, the homotopy pushout in~(\ref{FGextend}) implies that there is a homotopy cofibration 
\[\llnameddright{P^{4n-1}(m)}{G-q}{M_{2n}\vee S^{4n-1}}{I+H}{M}.\] 

\begin{lemma} 
   \label{I+Hinv} 
   The map 
   \(\llnamedright{\Omega(M_{2n}\vee S^{4n-1})}{\Omega(I+H)}{\Omega M}\) 
   has a right homotopy inverse. 
\end{lemma} 

\begin{proof} 
The plan is to use the right homotopy inverse for $\Omega(i+(\mathfrak{h}\circ H))$ in 
Lemma~\ref{i+frakhinv} and the naturality of Remark~\ref{GTcofibnat}. This will be done in steps. 
\medskip 

\noindent 
\textit{Step 1}. 
By~(\ref{M2n2}), $M_{2n}\simeq\Sigma A\vee P^{2n}(m)$. Let 
\(\namedright{\Sigma A}{a}{M_{2n}}\) 
be the inclusion of the wedge summand and recall that the composite 
\(\nameddright{\Sigma A}{a}{M_{2n}}{I}{M}\) 
is the definition of the map $j$ appearing in~(\ref{Fpo}), whose cofibre is the map 
\(\namedright{M}{\mathfrak{h}}{V}\). 
From this and the homotopy cofibration 
\(\llnameddright{P^{4n-1}(m)}{G-q}{M_{2n}\vee S^{4n-1}}{I+H}{M}\) 
we obtain a homotopy pushout diagram 
\[\diagram 
      P^{4n-1}(m)\rrdouble\dto^{i_{2}} & & P^{4n-1}(m)\dto^{G-q} & \\ 
      \Sigma A\vee P^{4n-1}(m)\rrto^-{a+(G-q)}\dto^{p_{1}} 
           & & M_{2n}\vee S^{4n-1}\rto^-{\overline{\mathfrak{h}}}\dto^{I+H} & V\ddouble \\ 
      \Sigma A\rrto^{j} & & M\rto^-{\mathfrak{h}} & V 
  \enddiagram\] 
where $i_{2}$ is the inclusion of the second wedge summand, $p_{1}$ is the pinch map 
onto the first wedge summand, and $\overline{\mathfrak{h}}$ is defined as $\mathfrak{h}\circ(I+H)$. 
\medskip 

\noindent 
\textit{Step 2}. By Lemma~\ref{preMhinv}, $\Omega\overline{\mathfrak{h}}$ has a 
right homotopy inverse 
\(s\colon\namedright{\Omega V}{}{\Omega(M_{2n}\vee S^{4n-1})}\). 
Let $s'$ be the composite 
\(s'\colon\llnameddright{\Omega V}{s}{\Omega(M_{2n}\vee S^{4n-1})}{\Omega(I+H)}{\Omega M}\). 
Then $s'$ is a right homotopy inverse for $\Omega\mathfrak{h}$ and there is a homotopy 
commutative diagram 
\begin{equation} 
  \label{ss'} 
  \diagram 
     \Omega V\rto^-{s}\ddouble & \Omega(M_{2n}\vee S^{4n-1})\dto^{\Omega(I+H)} \\ 
     \Omega V\rto^-{s'} & \Omega M. 
  \enddiagram 
\end{equation} 

\noindent 
\textit{Step 3}. 
The homotopy cofibration 
\(\nameddright{\Sigma A}{j}{M}{\mathfrak{h}}{V}\) 
and the existence of a right homotopy inverse~$s'$ for $\Omega\mathfrak{h}$ led to the 
identification of the homotopy fibre of $\mathfrak{h}$ as $(\Sigma\Omega V\wedge A)\vee\Sigma A$ 
via Theorem~\ref{GTcofib}. Similarly, the homotopy cofibration 
\(\llnameddright{\Sigma A\vee P^{4n-1}(m)}{a+(G-q)}{M_{2n}\vee S^{4n-1}}{\bar{\mathfrak{h}}}{V}\) 
and the existence of a right homotopy inverse~$s$ for $\Omega\overline{\mathfrak{h}}$ 
lets us use Theorem~\ref{GTcofib} to identify the homotopy fibre of $\overline{\mathfrak{h}}$ as 
$(\Sigma\Omega V\wedge(A\vee P^{4n-2}(m)))\vee(\Sigma A\vee P^{4n-1}(m))$. 
The compatibility of the $s$ and $s'$ in~(\ref{ss'}) lets us apply the naturality property 
in Remark~\ref{GTcofibnat} to obtain a homotopy fibration diagram 
\begin{equation} 
  \label{Vcompare} 
  \diagram 
       (\Sigma\Omega V\wedge(A\vee P^{4n-2}(m)))\vee(\Sigma A\vee P^{4n-1}(m)) 
            \rto\dto^{(\Sigma 1\wedge p_{1})\vee\Sigma p_{1}} 
          & M_{2n}\vee S^{4n-1}\rto^-{\overline{\mathfrak{h}}}\dto^{I+H} & V\ddouble \\  
       (\Sigma\Omega V\wedge A)\vee\Sigma A\rto  & M\rto^-{\mathfrak{h}} & V. 
  \enddiagram 
\end{equation} 

\noindent 
\textit{Step 4}. 
Finally, observe that the map $(\Sigma 1\wedge p_{1})\vee\Sigma p_{1}$ has a right 
homotopy inverse, and clearly the identity map on $V$ does as well. 
Since $\Omega(i+(\mathfrak{h}\circ H))$ has a right homotopy inverse by Lemma~\ref{i+frakhinv}
and it
factors as
 \[
 \Omega(i+(\mathfrak{h}\circ H)): \lllnameddright{\Omega(P^{2n}(m)\vee S^{4n-1})}{}{\Omega(M_{2n}\vee S^{4n-1})}{\Omega \overline{\mathfrak{h}}}{\Omega V},
 \]
$\Omega\overline{\mathfrak{h}}$ also has a right homotopy inverse.
Therefore,
Lemma~\ref{splittinglemma} implies that $\Omega(I+H)$ has a right homotopy inverse. 
\end{proof} 

Now we can prove Theorem~\ref{introI+Hfib}. 

\begin{proof}[Proof of Theorem~\ref{introI+Hfib}]  
From the homotopy cofibration 
\(\nameddright{P^{4n-1}(m)}{\mathfrak{G}}{M_{2n}\vee S^{4n-1}}{I+H}{M}\), 
where $\mathfrak{G}=G-q$, and the right homotopy inverse for $\Omega(I+H)$ in 
Lemma~\ref{I+Hinv}, the assertions follow immediately from Theorem~\ref{GTcofib}. 
\end{proof}

\section{An extension to some $2$-torsion cases I} 
\label{sec:general} 

In this section we consider a partial extension for parts~(a) and~(b) of Theorem~\ref{introloopMdecomp} 
to cases involving $2$-torsion. A full extension may not be possible due to issues involving Poincar\'{e} 
Duality as indicated by the lack of a $2$-primary analogue of Lemma~\ref{pkchoice}. 
Let $M$ be a $(2n-2)$-connected $(4n-1)$-dimensional Poincar\'{e} Duality 
complex such that $n\geq 2$ and 
\[
H^{2n}(M;\mathbb{Z})\cong\bigoplus_{k=1}^{\ell}\mathbb{Z}/p_{k}^{r_{k}}\mathbb{Z}\oplus\bigoplus_{s=1}^{t}\mathbb{Z}/2^{r_s}\mathbb{Z}
\]  
where each $p_{k}$ is an odd prime and $\ell\geq 1$. Then the $2n$-skeleton $M_{2n}$ 
of $M$ is homotopy equivalent to a wedge of Moore spaces  
\[
M_{2n}\simeq\bigvee_{k=1}^{\ell} P^{2n}(p_{k}^{r_{k}})\vee \bigvee_{s=1}^{t} P^{2n}(2^{r_{s}}).
\] 
Note the absence of mod-$2$ Moore spaces: this has to do with the smash product of 
two mod-$2$ Moore spaces as described in Remark~\ref{Fextend2remark}. 

As in the Introduction and Section \ref{sec:rankm}, let $m$ be the least common multiple of 
$\{p_{1}^{r_{1}},\ldots,p_{\ell}^{r_{\ell}}\}$ and let $m=\bar{p}_{1}^{\bar{r}_{1}}\cdots\bar{p}_{s}^{\bar{r}_{s}}$ 
be its prime decomposition. Notice that $\{\bar{p}_{1},\ldots,\bar{p}_{s}\}$ is the set of distinct 
primes in $\{p_{1},\ldots,p_{\ell}\}$  and each~$\bar{r}_{j}$ is the maximum power of $\bar{p}_{j}$ 
appearing in the list $\{p_{1}^{r_{1}},\ldots,p_{\ell}^{r_{\ell}}\}$.
Therefore $M_{2n}$ can be rewritten as 
\begin{equation} 
  M_{2n}\simeq P^{2n}(m)\vee\bigvee_{s=1}^{t} P^{2n}(2^{r_{s}})\vee\Sigma A  
\end{equation} 
where $\Sigma A$ is the wedge of the remaining Moore spaces in $M_{2n}$. 

Define $j$ and $j'$ by the composites  
\[j\colon \Sigma A\vee\bigvee_{s=1}^{t} P^{2n}(2^{r_{s}})\hookrightarrow\nameddright{P^{2n}(m)\vee\bigvee_{s=1}^{t} P^{2n}(2^{r_{s}})\vee\Sigma A}{\simeq}{M_{2n}}{}{M},\] 
\[j'\colon \Sigma A\hookrightarrow\nameddright{P^{2n}(m)\vee\bigvee_{s=1}^{t} P^{2n}(2^{r_{s}})\vee\Sigma A}{\simeq}{M_{2n}}{}{M}.\] 
Define the spaces $V$ and $V'$, and the maps $\mathfrak{h}$ and $\mathfrak{h}'$, by the homotopy pushout diagram 
\begin{equation} 
  \label{frakhpo} 
  \diagram 
      \Sigma A\rdouble\dto & \Sigma A\dto^{j'} & \\ 
      \Sigma A\vee\bigvee_{s=1}^{t} P^{2n}(2^{r_{s}})\rto^-{j}\dto & M\rto^-{\mathfrak{h}}\dto^{\mathfrak{h}'} 
            & V\ddouble \\ 
      \bigvee_{s=1}^{t} P^{2n}(2^{r_{s}})\rto & V'\rto & V. 
  \enddiagram  
\end{equation}  
Then $V=P^{2n}(m)\cup e^{4n-1}$ and $V'=\big(P^{2n}(m)\vee\bigvee_{s=1}^{t} P^{2n}(2^{r_{s}})\big)\cup e^{4n-1}$. 
Observe that the bottom row implies that there is a $p$-local homotopy equivalence $V\simeq V'$ for any 
odd prime $p$.  

We wish to show that $\Omega\frak{h}'$ has a right homotopy inverse. That is, the analogue 
of Theorem~\ref{introloopMdecomp} we aim to prove is based on a decomposition of $\Omega M$ 
involving $\Omega V'$ as a factor rather than $\Omega V$. To do so we will take a local-to-global 
approach by applying the fracture theorem of \cite[Theorem 8.1.3]{MP}. However, first we need a 
functional version of Lemma \ref{Fextend} and a modification of Proposition~\ref{Vmdecomp}. 

Let 
\(F\colon\namedright{S^{4n-2}}{}{M_{2n}}\) 
be the attaching map for the top cell of $M$.
Define $f$ and $f'$ by the composites
\[f\colon\nameddright{S^{4n-2}}{F}{M_{2n}}{\frak{q}}{P^{2n}(m)}\]
\[f'\colon\nameddright{S^{4n-2}}{F}{M_{2n}}{\frak{q}'}{P^{2n}(m)\vee\bigvee_{s=1}^{t} P^{2n}(2^{r_{s}})}\] 
where $\frak{q}$ and $\frak{q}'$ collapse $\Sigma A\vee\bigvee_{s=1}^{t} P^{2n}(2^{r_{s}})$ and $\Sigma A$ in $M_{2n}$ to a point respectively. Then $f$ and~$f'$ are the attaching maps for the top cell of $V$ and $V'$ respectively. In particular, there is a homotopy pushout 
\begin{equation} 
  \label{fpo} 
  \diagram 
     S^{4n-2}\rto^-{F}\ddouble & M_{2n}\rto\dto^{\mathfrak{q}} & M\dto^{\mathfrak{h}} \\ 
     S^{4n-2}\rto^-{f} & P^{2n}(m)\rto^-{i} & V  
  \enddiagram 
\end{equation} 
where $i$ is the inclusion of the $2n$-skeleton. 

Let $\widehat{m}$ be the least common multiple of 
$\{p_{1}^{r_{1}},\ldots,p_{\ell}^{r_{\ell}}\}\cup \{2^{r_{1}},\ldots,2^{r_t}\}$. 
In particular, $\widehat{m}=2^v m$ with $2^v={\rm max}\{2^{r_{1}},\ldots,2^{r_t}\}$. 
Anticipating that the upper bound on the exponent for $\pi_{\ast}(P^{2n}(2^{r}))$ in~\cite{Bar} 
is higher than for odd primes, let 
$\widetilde{v}=v+1$ and let $\widetilde{m}=2^{\widetilde{v}}m$. 
By~\cite[proof of Proposition 1.5]{N}, there is a canonical morphism of homotopy cofibrations
 \begin{equation} 
   \label{Q2t} 
   \diagram 
      S^{4n-2} \rto \ddouble &P^{4n-1}(\widetilde{m})\rto^-{\widetilde{q}} \dto^{Q} & S^{4n-1} \dto^{2^{\widetilde{v}}} \\ 
      S^{4n-2}  \rto              &P^{4n-1}(m) \rto^-{q} & S^{4n-1} ,
  \enddiagram 
\end{equation} 
where $Q$ collapses $P^{4n-1}(\widetilde{m})\simeq P^{4n-1}(2^{\widetilde{v}})\vee P^{4n-1}(m)$ to $P^{4n-1}(m)$ 
and $\widetilde{q}$ and $q$ are the pinch maps to the top cell.  
The following lemma is the analogue of Lemma \ref{Fextend}. 

\begin{lemma} 
   \label{Fextend2} 
 The maps 
   \(\namedright{S^{4n-2}}{F}{M_{2n}}\) and \(\namedright{S^{4n-2}}{f}{P^{2n}(m)}\) 
   extend to maps 
   \(G\colon\namedright{P^{4n-1}(\widetilde{m})}{}{M_{2n}}\) and \(g\colon\namedright{P^{4n-1}(m)}{}{P^{2n}(m)}\) respectively. Moreover, the extensions are compatible, that is, there is the homotopy commutative diagram   
     \[\diagram 
      P^{4n-1}(\widetilde{m})\rto^-{G} \dto^{Q} & M_{2n} \dto^{\frak{q}} \\ 
      P^{4n-1}(m) \rto^-{g} & P^{2n}(m).
  \enddiagram\] 
\end{lemma} 

\begin{proof} 
The existence of $G$ follows exactly as in the proof of Lemma~\ref{Fextend}, using the fact 
that~\cite{Bar} implies that $2^{r+1}\cdot\pi_{4n-2}(P^{2n}(2^{r}))\cong 0$ if $r\geq 2$.

A choice of the map $g$ is given by Lemma~\ref{Fextend}, but we need to make sure that 
a choice is made that also gives the asserted homotopy commutative diagram. Notice that 
there is a homotopy cofibration 
\(\nameddright{P^{4n-1}(2^{\widetilde{v}})}{\omega}{P^{4n-1}(\widetilde{m})}{Q}{P^{4n-1}(m)}\) 
where $\omega$ is the inclusion into $P^{4n-1}(\widetilde{m})\simeq P^{4n-1}(m)\vee P^{4n-1}(2^{\widetilde{v}})$. 
If the composite 
\[\namedddright{P^{4n-1}(2^{\widetilde{v}})}{\omega}{P^{4n-1}(\widetilde{m})}{G}{M_{2n}}{\mathfrak{q}}{P^{2n}(m)}\] 
is null homotopic then $\mathfrak{q}\circ G$ extends along $Q$ to a map 
\(g\colon\namedright{P^{4n-1}(m)}{}{P^{2n}(m)}\) 
and we are done. To see that $\mathfrak{q}\circ G\circ\omega$ is null homotopic, observe that 
it represents an element of $2$-torsion in $\pi_{4n-2}(P^{2n}(m))$. But the space $P^{2n}(m)$ 
is $2$-locally contractible since $m$ is a product of odd primes. 
\end{proof} 

\begin{remark} 
\label{Fextend2remark} 
It is the use of Lemma~\ref{Fextend} that prevents us from considering $2$-torsion 
in the cohomology of $M$. Its proof uses the property that the smash product 
$P^{a}(p^{r})\wedge P^{b}(p^{r})$ is homotopy equivalent 
to a wedge of two mod-$p^{r}$ Moore spaces: this only holds if $p^{r}\neq 2$. 
\end{remark} 

From the extension of $F$ to $G$ in Lemma~\ref{Fextend2} we obtain a homotopy cofibration diagram 
\begin{equation} 
  \label{FGextend2} 
  \diagram 
        S^{4n-2}\rto\ddouble & P^{4n-1}(\widetilde{m})\rto^-{\widetilde{q}}\dto^{G} & S^{4n-1}\dto^{H} \\ 
        S^{4n-2}\rto^-{F} & M_{2n}\rto^-{I} & M 
  \enddiagram 
\end{equation} 
where $I$ is the skeletal inclusion and $H$ is an induced map of cofibres. 

\begin{lemma} 
   \label{Hhpomap} 
   There is a homotopy commutative diagram 
   \[\diagram 
        S^{4n-1} \rto^-{H} \dto^{2^{\widetilde{v}}} & M \dto^{\frak{h}} \\ 
        S^{4n-1} \rto^-{h} & V 
    \enddiagram\] 
   for a map $h$ satisfying $h\circ q\simeq i\circ g$. 
\end{lemma} 

\begin{proof} 
Consider the cube 
\[\spreaddiagramrows{-1pc}\spreaddiagramcolumns{-1pc}\diagram
      P^{4n-1}(\widetilde{m})\rrto^-{\widetilde{q}}\drto^{Q}\ddto^{G} & & S^{4n-1}\dline\drto^{2^{\widetilde{v}}} & \\
      & P^{4n-1}(m)\rrto^-(0.3){q}\ddto^-(0.7){g} & \dto^-(0.3){H} & S^{4n-1}\dddashed|>\tip^-(0.6){h} \\
      M_{2n}\rline^-(0.6){I}\drto^{\mathfrak{q}} & \rto & M\drto^{\mathfrak{h}} & \\
      & P^{2n}(m)\rrto^-{i} & & V
 \enddiagram\] 
where the map $h$ will be defined momentarily. The top face is a homotopy pushout 
by~(\ref{Q2t}), the rear face homotopy commutes by~(\ref{FGextend2}), the left face homotopy 
commutes by Lemma~\ref{Fextend2}, and the bottom face homotopy commutes by~(\ref{fpo}). 
The homotopy commutativity of these four faces implies that 
$i\circ g\circ Q\simeq\mathfrak{h}\circ H\circ\widetilde{q}$. Therefore, as the top face is a homotopy 
pushout, there is a pushout map 
\(h\colon\namedright{S^{4n-1}}{}{V}\) 
such that $h\circ q\simeq i\circ g$ and $h\circ 2^{\widetilde{v}}\simeq\mathfrak{h}\circ H$. In particular, the 
homotopy $h\circ 2^{\widetilde{v}}\simeq\mathfrak{h}\circ H$ gives the homotopy commutative diagram 
asserted by the lemma. 
\end{proof} 

Next, we modify Proposition~\ref{Vmdecomp}. Similarly to the map $e$ in Section \ref{sec:rankm}, define $e'$ by the composite 
\[e'\colon\lllnameddright{\bigg(\prod_{j=1}^{s} S^{2n-1}\{\bar{p}_{j}^{\bar{r}_{j}}\}\bigg)\times\Omega S^{4n-1}} 
      {T\times\Omega h}{\Omega V\times\Omega V}{\mu}{\Omega V}.\]  
Notice that $e'$ replaces the map  $\frak{h}\circ H$ in the definition of $e$ appearing in Section~\ref{sec:rankm} by $h$, but the property from Lemma~\ref{Hhpomap} that $h\circ q\simeq i\circ g$ ensures that the argument for Proposition \ref{Vmdecomp} also applies to $e'$.
\begin{proposition} 
   \label{Vmdecomp2} 
   If $n\geq 2$ then the map $e'$ is a homotopy equivalence. ~$\qqed$
\end{proposition} 

Finally, we show that $\Omega\frak{h}'$ has a right homotopy inverse using a local-to-global approach.  
Let $T_o$ be the set of odd primes and $T_e=\{2\}$. 
\begin{lemma} 
   \label{localMhinv} 
   The map
   \(\namedright{\Omega M}{\Omega\frak{h}'}{\Omega V'}\) 
    has: 
    \begin{itemize} 
       \item[(i)] a $T_{o}$-local right homotopy inverse 
                     \(\theta_{o}\colon\namedright{\Omega V'}{}{\Omega M}\) and 
       \item[(ii)] a $T_{e}$-local right homotopy inverse 
                     \(\theta_{e}\colon\namedright{\Omega V'}{}{\Omega M}\), 
   \end{itemize} 
   both of whose rationalizations are the identity map on $\Omega S^{4n-1}$. 
\end{lemma} 
\begin{proof} 
For (i), by~(\ref{frakhpo}) the composite  
\(\nameddright{M}{\mathfrak{h}'}{V'}{}{V}\) 
is homotopic to  
\(\namedright{M}{\mathfrak{h}}{V}\). 
As 
\(\namedright{V'}{}{V}\) 
is a  $T_{o}$-local equivalence, to show that $\Omega\mathfrak{h}'$ has a $T_{o}$-local right 
homotopy inverse it suffices to prove that~$\Omega\frak{h}$ has a $T_{o}$-local right homotopy inverse 
\(\theta'_{o}\colon\namedright{\Omega V}{}{\Omega M}\). 
We then take $\theta_{o}$ to be the composite 
\(\nameddright{\Omega V'}{\simeq}{\Omega V}{\theta'_{o}}{\Omega M}\). 

Localize spaces and maps at $T_{o}$. Arguing as for Lemma~\ref{preMhinv} and using Lemma~\ref{Hhpomap} 
gives a homotopy commutative diagram
\begin{equation} 
  \label{Todgrm} 
  \diagram 
      &\Omega(P^{2n}(m)\vee S^{4n-1})  \rto^-{ \Omega(j'+H)} 
      & \Omega M \dto^-{\Omega\frak{h}} \\ 
     \bigg(\prod_{j=1}^{s} S^{2n-1}\{\bar{p}_{j}^{\bar{r}_{j}}\}\bigg)\times\Omega S^{4n-1}
       \rto^-{T\times\Omega h}   \urto^-{\mu\circ (S\times \Omega (\frac{1}{2^{\widetilde{v}}}))}    
       &\Omega V\times\Omega V \rto^-{\mu} &\Omega V
  \enddiagram 
\end{equation}  
while Proposition~\ref{Vmdecomp2} implies that the bottom row is the  homotopy equivalence $e'$. 
Therefore $\theta'_{o}=\Omega(j'+H)\circ\mu\circ(S\times\Omega(\frac{1}{2^{\widetilde{v}}}))\circ e'$ 
is a ($T_{o}$-local) right homotopy inverse for $\Omega\mathfrak{h}$. Rationally, $\mathfrak{h}$ 
is the identity map on $S^{4n-1}$, as is $h$ since $e'$ is an integral homotopy equivalence (technically, 
$h$ could have degree $\pm 1$ but if it is degree $-1$ we can replace $h$ by its negative). Thus 
the homotopy commutativity of~(\ref{Todgrm}) implies that, rationally, $\theta'_{o}$ must be the 
identity map on $\Omega S^{4n-1}$. 

For~(ii), the homotopy cofibration 
\(\nameddright{\Sigma A}{j'}{M}{\mathfrak{h}'}{V'}\) 
from~(\ref{frakhpo}) implies that $\mathfrak{h}'$ is a $T_{e}$-local homotopy equivalence 
since $\Sigma A$ is a wedge of odd primary Moore spaces and so is contractible when localized 
at $2$. Therefore $\mathfrak{h}'$ has a $T_{e}$-local right homotopy inverse $\theta'_{e}$. Further, 
as the rationalization of $\mathfrak{h}'$ is the identity map on $S^{4n-1}$, so is the rationalization 
of $\theta'_{e}$. Thus $\theta_{e}=\Omega\theta'_{e}$ is a $T_{e}$-local right homotopy inverse 
for $\Omega\mathfrak{h}'$ whose rationalization is the identity map on $\Omega S^{4n-1}$. 
\end{proof} 

The local right homotopy inverses for $\Omega\mathfrak{h}'$ in Lemma~\ref{localMhinv} are now 
assembled into an integral one. 

\begin{lemma} 
   \label{generalMhinv} 
The map \(\namedright{\Omega M}{\Omega\frak{h}'}{\Omega V'}\) has a right homotopy inverse $\theta$.
\end{lemma} 
\begin{proof}
By the fracture theorem of \cite[Theorem 8.1.3]{MP}, for any simply-connected space $X$ there 
is a homotopy pullback 
\[\diagram 
      X\rto\dto & X_{\mathbb{Q}}\dto^{\Delta} \\ 
      X_{T_{o}}\times X_{T_{e}}\rto^{r} & X_{\mathbb{Q}}\times X_{\mathbb{Q}} 
  \enddiagram\] 
where $X_{T_{o}}$, $X_{T_{e}}$ and $X_{\mathbb{Q}}$ are the $T_{o}$, $T_{e}$ and 
$\mathbb{Q}$-localizations of $X$ respectively, $r$ is rationalization and $\Delta$ is the 
diagonal map. In our case, consider the diagram 
\[\diagram 
     \Omega V'_{T_o} \times \Omega V'_{T_e} \rto^-{r} \dto^-{\theta_o\times \theta_e} &
     \Omega V'_{\mathbb{Q}} \times \Omega V'_{\mathbb{Q}}  \dto^-{\theta_{\mathbb{Q}}\times \theta_{\mathbb{Q}}}& 
     \Omega V'_{\mathbb{Q}} \lto_-{\Delta} \dto^-{\theta_{\mathbb{Q}}}\ \\
      \Omega M_{T_o} \times \Omega M_{T_e} \rto^-{r}  &
     \Omega M_{\mathbb{Q}} \times \Omega M_{\mathbb{Q}}  & 
     \Omega M_{\mathbb{Q}} \lto_-{\Delta} 
  \enddiagram\] 
where $\theta_o$ and $\theta_e$ respectively are the $T_{o}$ and $T_{e}$-local right homotopy inverses 
for $\Omega\mathfrak{h}'$ in Lemma \ref{localMhinv} and $\theta_{Q}$ is the rationalization  
of the identity map on $\Omega  S^{4n-1}$. The left square homotopy commutes by 
Lemma~\ref{localMhinv} and the right square commutes by the naturality of the diagonal map. 
By the fracture theorem, the homotopy pullback of the maps in the top row is $\Omega V'$ and the 
homotopy pullback of the maps in the bottom row is $\Omega M$. The pullback property for $\Omega M$ 
and the homotopy  commutativity of the two squares implies that there is a pullback map 
\(\theta\colon\namedright{\Omega V'}{}{\Omega M}\) 
with the property that its $T_{e}$-localization is $\theta_{e}$, its $T_{o}$-localization is $\theta_{o}$ 
and its rationalization is $\theta_{\mathbb{Q}}$. Thus $\theta$ is a right homotopy inverse 
for $\Omega\mathfrak{h}'$ because it is when localized at any prime or rationally. 
\end{proof} 

From the homotopy cofibration 
\(\nameddright{\Sigma A}{j'}{M}{\frak{h}'}{V'}\) 
and the right homotopy inverse $\theta$ of $\Omega\frak{h}'$ in Lemma~\ref{generalMhinv}, 
the following theorem follows immediately from Theorem~\ref{GTcofib}. 

\begin{theorem} 
   \label{genloopMdecomp} 
   Let $M$ be a $(2n-2)$-connected $(4n-1)$-dimensional Poincar\'{e} Duality 
complex such that $n\geq 2$ and 
\[
H^{2n}(M;\mathbb{Z})\cong\bigoplus_{k=1}^{\ell}\mathbb{Z}/p_{k}^{r_{k}}\mathbb{Z}\oplus\bigoplus_{s=1}^{t}\mathbb{Z}/2^{r_s}\mathbb{Z}
\]  
where each $p_{k}$ is an odd prime, each $r_s\geq 2$, and $\ell\geq 1$. Then with $V'$ and $A$ chosen as above: 
   \begin{letterlist} 
      \item there is a homotopy fibration 
               \[\llnameddright{(\Sigma\Omega V'\wedge A)\vee\Sigma A}{[\gamma,j']+j'}{M}{\frak{h}'}{V'}\] 
               where $\gamma$ is the composite 
               \(\gamma\colon\nameddright{\Sigma\Omega V'}{\Sigma \theta}{\Sigma\Omega M}{ev}{M}\); 
      \item the homotopy fibration in~(a) splits after looping to give a homotopy equivalence 
               \[\Omega M\simeq\Omega V'\times\Omega((\Sigma\Omega V'\wedge A)\vee\Sigma A).\]  
   \end{letterlist} 
\end{theorem} 

Note that when $t=0$, Theorem \ref{genloopMdecomp} reduces to part (a) and (b) of Theorem \ref{introloopMdecomp}. Note also that, unlike Theorem~\ref{introloopMdecomp}, 
Theorem~\ref{genloopMdecomp} does not decompose $\Omega V'$ any further.

\section{An extension to some $2$-torsion cases II} 
\label{sec:p=2} 

Finally, we consider an extension for part~(c) of Theorem~\ref{introloopMdecomp} to  
certain special cases involving $2$-torsion. 
In general, when $V=P^{2n}(2^{r})\cup e^{4n-1}$ it is unreasonable to expect a decomposition 
$\Omega V\simeq S^{2n-1}\{2^{r}\}\times\Omega S^{4n-1}$ since this implies that the 
space $S^{2n-1}\{2^{r}\}$ is an $H$-space. Often this is not the case, for example, 
if $n=3$ or $n\geq 5$ then $S^{2n-1}\{2\}$ is not an $H$-space~\cite{C2}. A full classification 
of when $S^{2n-1}\{2^{r}\}$ is an $H$-space seems not to appear in the literature. 
However, by~\cite[Corollary 21.6]{C1} it is known that $S^{3}\{2^{r}\}$ is an $H$-space 
if $r\geq 3$ and $S^{7}\{2^{r}\}$ is an $H$-space if~$r\geq 4$. In these cases we show that the 
arguments in Section~\ref{sec:rank1} hold, giving a decomposition of~$\Omega V$. 

Lemma~\ref{tildefimage} and Proposition~\ref{hlgyloopV} were proved for all primes $p$. 
The first point in Section~\ref{sec:rank1} where the restriction $p\geq 3$ occurred was 
in the the existence of the extension $g$ for $f$ in~(\ref{fgextend}). In general, it may 
not be the case that $2^{r}\cdot\pi_{4n-2}(P^{2n}(2^{r}))\cong 0$. However, Sasao~\cite{Sa} 
showed that $2^{r}\cdot\pi_{6}(P^{4}(2^{r}))\cong 0$ if $r\geq 3$ and 
$2^{r}\cdot\pi_{14}(P^{8}(2^{r}))\cong 0$ if $r\geq 4$. Thus in these cases 
we obtain a homotopy cofibration diagram as in~(\ref{fgextend}). The argument 
for Lemma~\ref{oddph} now goes through in exactly the same manner. The 
maps $s$, $t$ and $e$ following Lemma~\ref{oddph} were defined for all primes $p$, 
and the restriction to odd primes in Proposition~\ref{loopVdecomp} was present 
only to: (i) invoke Lemma~\ref{oddph} and (ii) in the $n=2$ case, ensure that the composite 
\(\nameddright{S^{6}}{f}{P^{4}(2^{r})}{q}{S^{4}}\) 
is null homotopic so that there is an extension of $q$ to a map 
\(\namedright{V}{}{S^{4}}\). 
Therefore Proposition~\ref{loopVdecomp} will hold: (i) for $n=4$ and $r\geq 4$, and (ii) for 
$n=2$ and $r\geq 3$ with the extra assumption that there is a map 
\(\namedright{V}{}{S^{4}}\) 
inducing a surjection in mod-$2$ homology. 

\begin{proposition} 
   \label{p=2decomp} 
   Let $V=P^{2n}(2^{r})\cup e^{4n-1}$ be a Poincar\'{e} Duality complex. 
   \begin{letterlist} 
      \item If $n=2$, $r\geq 3$ and there is a map 
               \(\namedright{V}{}{S^{4}}\) 
               inducing a surjection in mod-$2$ homology, then there is a homotopy equivalence  
               $\Omega V\simeq S^{3}\{2^{r}\}\times\Omega S^{7}$; 
      \item if $n=4$ and $r\geq 4$ then there is a homotopy equivalence 
               $\Omega V\simeq S^{7}\{2^{r}\}\times\Omega S^{15}$. 
   \end{letterlist} 
\end{proposition} 
\vspace{-1cm}~$\qqed$\bigskip 

For example, if $\tau(S^{2n})$ is the unit tangent bundle of $S^{2n}$ then, as a $CW$-complex, 
$\tau(S^{2n})=P^{2n}(2)\cup e^{4n-1}$, and there is a fibration 
\(\nameddright{S^{2n-1}}{}{\tau(S^{2n})}{}{S^{2n}}\). 
For $r\geq 2$, define the ``mod-$2^{r}$ tangent bundle" by the homotopy pullback 
\[\diagram 
       S^{2n-1}\rto\ddouble & \tau_{r}(S^{2n})\rto\dto & S^{2n}\dto^{\underline{2}^{r-1}} \\ 
       S^{2n-1}\rto & \tau(S^{2n})\rto & S^{2n} 
  \enddiagram\] 
where $\underline{2}^{r-1}$ is the map of degree~$2^{r-1}$. As a $CW$-complex, 
$\tau_{r}(S^{2n})=P^{2n}(2^{r})\cup e^{4n-1}$ and $\cohlgy{\tau_{r}(S^{2n})}$ satisfies 
Poincar\'{e} Duality. Proposition~\ref{p=2decomp} implies that there are homotopy equivalences 
$\Omega\tau_{r}(S^{4})\simeq S^{3}\{2^{r}\}\times\Omega S^{7}$ if $r\geq 3$ and 
$\Omega\tau_{r}(S^{8})\simeq S^{7}\{2^{r}\}\times\Omega S^{15}$ if $r\geq 4$. 

\begin{remark} 
The argument for Proposition~\ref{p=2decomp} is independent of 
prior knowledge that $S^{3}\{2^{r}\}$ for $r\geq 3$ or $S^{7}\{2^{r}\}$ for $r\geq 4$ 
are $H$-spaces. So the loop space decompositions of the mod-$2^{r}$ tangent 
bundles is a new proof of this property, since the retractions of $S^{3}\{2^{r}\}$ for $r\geq 3$ 
and $S^{7}\{2^{r}\}$ for $r\geq 4$ off loop spaces imply that they are $H$-spaces. 
The previous argument in~\cite{C1} examined the $H$-deviation of the degree~$2^{r}$ map. 
\end{remark} 

Generalizing to the case $V=P^{2n}(2m)\cup e^{4n-1}$ where $m$ is divisible by 
more than one prime seems to be much more difficult. Our argument breaks down with 
the loss of Lemma~\ref{pkchoice}. It would be interesting to know if a different argument can 
be used to make progress. 
\bigskip 

\noindent 
\textsc{Conflict of Interest Statement}: On behalf of all authors, the corresponding author states that 
there is no conflict of interest.

\bibliographystyle{amsalpha}

\end{document}